\documentclass[10pt,oneside]{article}

\usepackage{enumerate,amsthm,amsfonts, mathrsfs, amssymb, latexsym}
\usepackage{makeidx,setspace, fancyhdr}
\usepackage{todonotes}
\usepackage{amsmath}
\usepackage{verbatim}
\usepackage{graphicx}
\usepackage{cancel}
\usepackage{dsfont}
\usepackage[normalem]{ulem}

\usepackage{amsthm}
\usepackage{amssymb}
\usepackage{epstopdf}
\usepackage{marvosym}
\usepackage{enumitem}
\usepackage{hyperref}

\usepackage[left=2.5cm,right=2.5cm,top=2.5cm,bottom=2.5cm,headsep=1.5cm]{geometry}

\usepackage{soul,xcolor}

\newcommand{\R}{\ensuremath{\mathbb{R}}}
\newcommand{\T}{\ensuremath{\mathbb{T}}}
\newcommand{\N}{\ensuremath{\mathbb{N}}}
\newcommand{\E}{\ensuremath{\mathbb{E}}}

\newcommand{\e}{\ensuremath{\epsilon}}

\newcommand{\ud}{\,\mathrm{d}}

\theoremstyle{plain}

\newtheorem{theorem}{Theorem}
\newtheorem{corollary}[theorem]{Corollary}

\newtheorem{lemma}[theorem]{Lemma}

\theoremstyle{definition}
\newtheorem{definition}[theorem]{Definition}
\newtheorem{hyp}[theorem]{Assumptions}

\newtheorem*{theorem*}{Theorem}

\theoremstyle{remark}
\newtheorem{remark}[theorem]{Remark}
\newtheorem{ex}[theorem]{Example}

\title{Regularized vortex approximation for 2D Euler equations with transport noise}

\author{
	Michele Coghi \thanks{
		Institut f\"ur Mathematik, Technische Universit\"at Berlin,
		Stra\ss e des 17. Juni 136, 
		10623 Berlin, Germany
	}, 
	Mario Maurelli \thanks{
		Dipartimento di Matematica `Federigo Enriques', Universit\`a degli Studi di Milano, via Saldini 50, 20133 Milano, Italy
	}
}
\date{}

\begin{document}

\maketitle

\begin{abstract}
	We study a mean field approximation for the 2D Euler vorticity equation driven by a transport noise. We prove that the Euler equations can be approximated by interacting point vortices driven by a regularized Biot-Savart kernel and the same common noise. The approximation happens by sending the number of particles $N$ to infinity and the regularization $\epsilon$ in the Biot-Savart kernel to $0$, as a suitable function of $N$.
\end{abstract}

\section{Introduction}

In this paper we consider the stochastic Euler equations on the two-dimensional torus $\mathbb{T}^2$, in vorticity form, driven by transport noise, namely
\begin{align}
\begin{aligned}\label{eq:vort_intro}
\partial_t \xi +u\cdot\nabla \xi +\sum_k \sigma_k\cdot\nabla \xi \circ \dot{W}^k = 0,
\qquad
u=K\star \xi,
\end{aligned}
\end{align}
where $\xi=\xi(t,x,\omega)$ is the unknown vorticity, $K$ is the Biot-Savart kernel, $\sigma_k$ are given, divergence-free vector fields, satisfying certain assumptions, $W^k$ are independent real Brownian motions and $\circ$ denotes Stratonovich integration. We prove convergence, with quantitative bounds, of a system of point vortices, with regularized kernel $K^\e $, to the bounded solution $\xi$ to \eqref{eq:vort_intro}.

The deterministic Euler equations describe the motion of an incompressible, non-viscous fluid; in two dimensions one can use the equivalent vorticity formulation, that is \eqref{eq:vort_intro}, where $u=u(t,x)$ represents the velocity of the fluid at time $t$ and space $x$ and $\xi(t,x)=\operatorname{curl}u(t,x)$ is its vorticity. In 2D, well-posedness holds among bounded solution, as proved in \cite{Yud1963}, see also \cite[Section 2.3]{MarPul1994} for an alternative proof. Concerning the stochastic Euler equations, there are various results depending on the type on noise; the transport noise in the vorticity, which we consider here in \eqref{eq:vort_intro}, is motivated by the transport nature of the vorticity equation. For equation \eqref{eq:vort_intro} in 2D, existence and uniqueness of (probabilistically) strong, bounded solutions is proved in \cite{Brzezniak_Flandoli_Maurelli}, see also \cite{CriFlaHol2019} and \cite{HofLeaNil2019+} for resp. a 3D analogue and a rough path analogue of \eqref{eq:vort_intro}. Transport noise has also been used to show regularization by noise phenomena, mostly for the linear case (\cite{FlaGubPri2010,beck2014stochastic} and several other works), though isolated nonlinear examples also exist (see e.g. \cite{DelFlaVin2014,FlaGubPri2011,GessMaurelli,FlaLuo2019} and the recent review \cite{BiaFla2020}).

The vortex approximation is an approximation of the solution to the Euler equations in vorticity form via the weighted empirical measure of a system of interacting diffusions. The idea is formally as follows: Take a weighted empirical measure $\frac{1}{N}\sum_{i=1}^N \xi^{i,N} \delta_{X^{i,N}_0}$ which approximates the initial condition $\xi_0$ and consider the following system of interacting diffusions:
\begin{align}
dX^{i,N}_t = \frac{1}{N} \sum_{j\neq i} \xi^{j,N} K(X^{i,N}_t-X^{j,N}_t) dt +\sum_k \sigma_k(X^{i,N}_t) \circ dW^k_t,\quad i=1,\ldots N.\label{eq:vortex_intro}
\end{align}
Then, formally and ignoring self interaction (that is, assuming formally $K(0)=0$), the empirical measure $\frac{1}{N}\sum_{i=1}^N \xi^{i,N} \delta_{X^{i,N}_t}$ is a solution to \eqref{eq:vort_intro} in the distributional sense. Hence we might expect, by a continuity argument with respect to the initial condition, that $\frac{1}{N}\sum_{i=1}^N \xi^{i,N} \delta_{X^{i,N}_t}$ approximates the solution $\xi_t$. The system \eqref{eq:vortex_intro} describes the motion of interacting vortices and is similar to the system of interacting diffusions approximating a McKean-Vlasov SDE, see e.g. \cite{Szn1984}, with one important difference: the vortices in \eqref{eq:vortex_intro} are driven by the space correlated noise $\sum_k \sigma_k(x) \circ dW^k_t$, while in the McKean-Vlasov SDE approximation, the particles are driven by independent Brownian motions. In the case of independent Brownian motions as driving signals, the limit of the empirical measures is expected to solve the deterministic Navier-Stokes equations (in vorticity form) rather than the stochastic Euler equations, as noted by Chorin \cite{chorin1982}. The vortex system can be also viewed as a discrete approximation of the Euler equations; other discrete models in stochastic fluid dynamics are the shell models and the dyadic models (see e.g. \cite{BesGarSch2016,Bia2013,BarBiaFla2013,BiaMor2017}).

When coming to a rigorous proof of the above convergence argument, two difficulties arise: 1) the interaction kernel $K$ is irregular, precisely we expect $K(x) \approx x^\perp/|x|^2$ close to $0$; 2) the noise term prevents us from exploiting classical continuity arguments used in the deterministic context. In the deterministic context, for regular interaction kernel, convergence of the particle system is proved in \cite{Do}. The case of 2D Euler equations is considered in \cite{MarPul1982}: the authors consider a system of interacting vortices under a regularized kernel $K^\e $ and prove the convergence of this system to the Euler equations, assuming the convergence of the initial positions at rate $\zeta_N$ and tuning the regularization parameter $\e =\e (N)$ as a suitable, double logarithmic function of $\zeta_N$ \cite[Theorem 4.1]{MarPul1982}; see also \cite[Section 5.3]{MarPul1994}. The convergence of the original vortex system \eqref{eq:vortex_intro} (without noise), with no regularization, is proved in \cite{GooHouLow1990} and also in \cite{Sch1996}, in the latter paper also for unbounded solutions to \eqref{eq:vort_intro} (without noise), via a suitable randomization of the initial conditions of the vortex system. The paper \cite{Fla2018} shows the approximation result for distributional solutions under the white noise invariant measure $\mu$ on $\mathbb{T}^2$: precisely, if the initial conditions $X^{i,N}_0$ are taken independent and identically distributed with uniform law and the intensities $\xi^{j,N}$ are taken i.i.d. $\mathcal{N}(0,N)$, then the vortex system converges a.s. to a random, stationary solution to the Euler equations with one-time marginals distributed as $\mu$; the result is generalized also to solutions whose one-time marginals are absolutely continuous with respect to $\mu$. For other convergence results in the deterministic case the reader can refer to \cite{geldhauser2018limit, FlandoliSaal, grotto2019central}. 

In the stochastic case, \cite{coghi_flandoli} proves the convergence of the particle system \eqref{eq:vortex_intro} for a regular kernel $K$ and a non-negative initial distribution $\xi_0$. To our knowledge, the only paper dealing with approximation of stochastic Euler equations via vortices is \cite{FlaLuo2019}, where the analogue result of \cite{Fla2018} for the stochastic case is proved (the authors prove also an improved, compared to the deterministic case, regularity of the density with respect to $\mu$). The paper \cite{FlaGubPri2011} shows that, for any fixed $N$, the vortex system \eqref{eq:vortex_intro} is well-posed for every initial condition, at least for suitably non-degenerate $\sigma_k$, while the corresponding deterministic system can collapse for special (zero Lebesgue measure) initial conditions.

Note that, in the case of independent noises $dW^i$, that is, the case of deterministic Navier-Stokes as expected limiting equation, better results of convergences (in terms of rates and larger class of initial conditions) can be proved, see \cite{FouHauMis2014} and \cite{JabWan2018} as two remarkable examples. Note also that, in the 3D case, an analogue approximation has been proposed, for the deterministic Euler equations, in \cite{BesCogFla2017,BesCogFla2019}, replacing vortex points by vortex filaments.

In this paper we show the vortex approximation for bounded solutions to the stochastic 2D Euler equation \eqref{eq:vort_intro}, using a vortex system with regularized kernel $K^\e $, namely
\begin{align}\label{eq:part_reg_intro}
dX^{i,N}_t = \frac{1}{N} \sum_{j\neq i} \xi^{j,N} K^\e (X^{i,N}_t-X^{j,N}_t) dt +\sum_k \sigma_k(X^{i,N}_t) \circ dW^k_t,\quad i=1,\ldots N,
\end{align}
for a regularization parameter $\e =\e (N)$. Our main result is

\begin{theorem*}[see Theorem \ref{thm: convergence_euler}]
	Assume that $\sigma_k$ are sufficiently regular, and let $\xi_0$ be a bounded initial vorticity. Assume that $\frac{1}{N}\sum_{i=1}^N \xi^{i,N} \delta_{x^i}$ converges to $\xi_0$ with rate $\zeta_N$, as $N\rightarrow \infty$. Let $\e (N) \approx ( - \log \zeta_N)^{- \delta}$ for a suitable $\delta \in \R_{+}$ and let $X^{i,N}$ be the solution to the regularized vortex system \eqref{eq:part_reg_intro} with initial condition $(x^1, \dots, x^N)$. Then the path of empirical measures $(\frac{1}{N}\xi^{i,N}\sum_{i=1}^N \delta_{X^i_t})_t$ converges in $W^{1,\infty}(\mathbb{T}^2)^*$, as $N\rightarrow\infty$, to the (unique) bounded solution to the stochastic 2D Euler equations \eqref{eq:vort_intro}.
\end{theorem*}

We use the strategy of \cite{MarPul1982} applied to the stochastic case. Note that, by a technical trick in the fixed point argument in Section \ref{Regular_Section} (see Remark \ref{rmk:dist_c}), we can deal with $\e (N)$ as logarithmic function of $\zeta_N$, rather than double logarithmic as in \cite{MarPul1982}, though we expect the result to be non-optimal, as for the deterministic case. We leave the investigation of the convergence of the true vortex system \eqref{eq:vortex_intro} for future research.

The paper is organized as follows. In Section \ref{Regular_Section}, we deal with the convergence of \eqref{eq:part_reg_intro} to the regularized version of the Euler equation \eqref{eq:vort_intro} (replacing the Biot-Savart kernel $K$ with $K^{\e }$, for fixed $\e $). We use the techniques in \cite{coghi_flandoli}, showing in addition the convergence in $L^p$ in the $\omega$ variable for every $p\geq 1$ and accounting for non positive measures as well.
Then, in Section \ref{Euler_Section}, we deal with the convergence of the solution of the regularized Euler equation as the regularization parameter $\e $ tends to $0$; we use the techniques in \cite{Brzezniak_Flandoli_Maurelli}, showing in addition convergence in $L^p(\Omega;C([0,T];L^1(\mathbb{T}^2)))$ for $p\geq 1$ (in \cite{Brzezniak_Flandoli_Maurelli}, convergence is shown only in $C([0,T];L^1(\Omega\times\mathbb{T}^2))$). This is shown in Theorem \ref{convergence_approximation}. Finally, in Theorem \ref{thm: convergence_euler} we prove that the empirical measure of the system \eqref{eq:part_reg_intro} converges to the solution of the Euler equation \eqref{eq:vort_intro}.

\subsubsection*{Acknowledgements}
The authors are very thankful to Prof. Franco Flandoli for suggesting the idea and many fruitful discussions on the topic. \\
Support from the Hausdorff Research Institute for Mathematics in Bonn under the Junior Trimester Program `Randomness, PDEs and Nonlinear Fluctuations' is gratefully acknowledged.

\section{Preliminaries}\label{preliminaries}
\subsection{Spaces of measures}

We start with some notations used throughout the paper. Given a compact metric space $(E,d)$ (in practice, $E=\mathbb{T}^2$ with the Euclidean distance), we call $\mathcal{M}(E)$ the space of finite, signed Borel measures on $E$. Given a measurable function $\varphi:E\to\mathbb{R}$ and a measure $\mu\in\mathcal{M}(E)$, we write
\begin{equation*}
\mu(\varphi):=\int_E\varphi(x)\mu(dx).
\end{equation*}
We call $C_b(E)$ the space of continuous bounded functions on $E$, endowed with the supremum norm $\Vert \varphi \Vert_{\infty}=\sup_{x\in E}\vert\varphi(x)\vert$. The space $\mathcal{M}(E)$, being the dual of $C_b(E)$, is naturally endowed with the dual norm
\begin{equation*}
\Vert\mu\Vert:=\sup_{\Vert\varphi\Vert_{\infty}\leq1}\vert\mu(\varphi)\vert.
\end{equation*}
Given a finite signed Borel measure $\mu$, we denote by $|\mu|$ its variation measure (it holds $\Vert\mu\Vert = |\mu|(X)$).

The space of bounded Lipschitz continuous functions on $E$ will be called $BL(E)$, while the unit ball in this space is
\begin{equation*}
BL_1(E):=\left\{\varphi\in Lip(E)\quad |\quad \Vert \varphi \Vert_{\infty}+\mbox{Lip}(\varphi) \leq 1\; \right\},
\end{equation*}
where $\mbox{Lip}(\varphi):=\sup_{x,y\in E}\frac{\vert \varphi(x)-\varphi(y)\vert}{\vert x-y \vert}$.

Now we endow $\mathcal{M}(E)$ with the Kantorovich-Rubinstein (or $1$-Wasserstein) metric
\begin{equation*}
W_1(\mu,\nu):= \sup_{\varphi\in BL_1(E)}\left\vert \mu(\varphi)-\nu(\varphi)\right\vert.
\end{equation*}
The space $\mathcal{M}(E)$ is not complete with respect to this metric. However, for every $M > 0$, the closed ball in the total variation norm $\mathcal{M}_M(E) := \{ \mu \in \mathcal{M} \mid \| \mu \| \leq M \}$ is complete with respect to $W_1$.

We call $\mathcal{P}(E)$ the space of probability measures on $E$.

\begin{remark}
	The fact that $\mathcal{M}_M(E) $ is closed under $W^1$ is classical, we give here a short proof. Let $(\mu^n)_{n\in\N} \in \mathcal{M}_M(E)$ be a sequence converging to $\mu$ in $W^1$. For every $\varphi \in BL(E)$, we have $\lim_{n\to \infty} \mu^n(\varphi) - \mu( \varphi ) \leq  \lim_{n\to \infty} \| \varphi \|_{BL} W_1 ( \mu^n, \mu) = 0$.
	Hence, since the Lipschitz functions are dense in the continuous functions, we have
	\begin{equation*}
	\sup_{\varphi \in C_b(E), \| \varphi \|_{\infty} \leq 1 } | \mu( \varphi ) |
	\leq
	\sup_{\varphi \in BL(E), \| \varphi \|_{\infty} \leq 1 } | \mu( \varphi ) |
	\leq 
	\sup_{\varphi \in BL(E), \| \varphi \|_{\infty} \leq 1 } \sup_{n} | \mu^n( \varphi ) |
	\leq
	M.
	\end{equation*}
\end{remark}

Let $(\Omega, \mathcal{F}, (\mathcal{F}_t)_{t\geq0},\mathbb{P})$ be a filtered probability space, satisfying the standard assumption (that is, completeness and right-continuity). Fix a time horizon $T>0$ and a real number $p \in [1,\infty)$, we define the space $V^{p,T}_M := L^p(\Omega;C([0,T], \mathcal{M}_M( \mathbb{T}^2 )))$ of $(\mathcal{F}_t)_{t\geq0}$-progressively measurable stochastic processes endowed with the distance 
\begin{equation*}
d_p(\mu,\nu):=\mathbb{E}\left[\sup_{t\in[0,T]}W_1(\mu_t,\nu_t)^p \right]^{\frac1p} .
\end{equation*}

\begin{remark}
	The space $V^{p,T}_M$ with the distance $d_p$ is complete, we give a short proof for completeness. Indeed, given a Cauchy sequence $(\mu_n)_{n\in \N} \subset V^{p,T}_M$ there is a subsequence $(\mu_{ n_k } )_{k\in \N}$ which is almost surely a Cauchy sequence in
	$  C([0,T], \mathcal{M}_M(\mathbb{T}^2))$. Since $C([0,T], \mathcal{M}_M(\mathbb{T}^2))$ is complete, there exists a null set $N \subset \Omega$, such that, for all $\omega \in N^c$, there exists $\mu(\omega) \in C([0,T], \mathcal{M}_M(\mathbb{T}^2))$ such that $\sup_{ t \in [0,T]} W_1(\mu_{ n_k } ( \omega ), \mu (\omega) ) \to 0$ as $k\to\infty$. Adaptedness of $\mu$ follows from adaptedness of $\mu_{n_k}$. Since the distance is bounded, dominated convergence concludes the argument.
\end{remark}

For later convenience, given a positive constant $c>0$, we define the distance
\begin{equation}
d_p^c(\mu,\nu) = \mathbb{E}\left[ \sup_{t\in [0,T]}\left(e^{-ct}W_1(\mu_t,\nu_t)^p \right) \right]^\frac1p .\label{eq:dist_c}
\end{equation}
Note that, for every $p \in [1,\infty)$ and $ c > 0 $, the two distances $d_p$ and $d_p^c$ are equivalent.
We will sometimes use the short notation $L^p_x$ to mean $L^p(\mathbb{T}^2)$. 

\begin{remark}\label{Wasserstein_bound}
	The distance $W^1$ has the following property: for any $\mu$ in $\mathcal{M}_M(E)$, for every two Borel maps $f,g:E\rightarrow E$, it holds
	\begin{equation*}
	W_1(f_\# \mu,g_\# \mu) \le \|\mu\| \|f-g\|_{\infty}.
	\end{equation*}
	Indeed, for every $\varphi$ in $BL_1(E)$, we have
	\begin{equation*}
	|f_\# \mu(\varphi) -g_\# \mu(\varphi)| = |\mu(\varphi(f)-\varphi(g))|\le \|\mu\| \|f-g\|_{\infty}.
	\end{equation*}
	For the distance $d_p^c$ a similar property holds: for any $\mu$ in $\mathcal{M}_M(\mathbb{T}^2)$, for every two measurable maps $f,g:[0,T]\times \mathbb{T}^2 \times \Omega \rightarrow \mathbb{T}^2$, it holds
	\begin{equation*}
	d_p^c(f_\# \mu,g_\# \mu) \le \|\mu\| \sup_{x\in\mathbb{T}^2} \mathbb{E}\left[ \sup_{t\in [0,T]} \left(e^{-ct}|f_t(x)-g_t(x)| \right)^p \right]^\frac1p.
	\end{equation*}
	Indeed, recalling that $|\mu(\psi)|^p\le \|\mu\|^{p-1}\int |\psi|^p d|\mu|$ for every $\psi$,
	\begin{align*}
	&d_p^c(f_\# \mu,g_\# \mu)^p = \mathbb{E}\left[ \sup_{t\in [0,T]} \sup_{\varphi\in BL_1(\mathbb{T}^2)} \left(e^{-ct}|\mu(\varphi(f_t)-\varphi(g_t))| \right)^p \right]\\
	&\le \|\mu\|^{p-1} \int_{\mathbb{T}^2} \mathbb{E}\left[ \sup_{t\in [0,T]} \sup_{\varphi\in BL_1(\mathbb{T}^2)} \left(e^{-ct}|\varphi(f_t)-\varphi(g_t)| \right)^p \right] d|\mu|(dx)\\
	&\le \|\mu\|^{p-1} \int_{\mathbb{T}^2} \mathbb{E}\left[ \sup_{t\in [0,T]} \left(e^{-ct}|f_t-g_t| \right)^p \right] d|\mu|(dx)\\
	&\le \|\mu\|^p \sup_{x\in\mathbb{T}^2}  \mathbb{E}\left[ \sup_{t\in [0,T]} \left(e^{-ct}|f_t(x)-g_t(x)| \right)^p \right].
	\end{align*}
\end{remark}

\subsection{The noise}

Here we give the assumptions on the noise. In the following, $\sigma_k :\mathbb{T}^2\to\R^2$ is a vector field, for every $ k \in \N $ and $Q:\mathbb{T}^{2}\times\mathbb{T}^{2}\rightarrow\mathbb{R}^{2\times 2}$ is the space covariance (matrix-valued) function defined by
\begin{equation*}\label{definition_Q}
Q^{ij}\left(  x,y\right)  :=\sum_{k=1}^{\infty}\sigma_{k}^{i}\left(  x\right)
\sigma_{k}^{j}\left(  y\right).
\end{equation*}
\begin{hyp}
	\label{noise}
	\begin{enumerate}
		\item[i)] $\sigma_k:\mathbb{T}^{2}\to\mathbb{R}^{2}$ are $C^2$ functions satisfying
		$
			\sum_{k=1}^\infty \|\sigma_k\|_{C^2} <\infty
		$.
		\item[ii)]  $\sigma_k$ are divergence free vector fields, i.e. 
		$
		\operatorname{div}\sigma_k=0, \quad \forall k\geq 1
		$.
		\item[iii)] The covariance function $:\mathbb{T}^{2}\rightarrow\mathbb{R}^{2\times 2}$ satisfies
		\begin{enumerate}
			\item $Q(  x,y)  =Q(  x-y)$ (space homogeneity of the random field $\sum_{k=1}^{\infty}\sigma_{k}\left(  x\right)  B_{t}^{k}$);
			\item $Q\left(0\right)  = a I$ for some $a\ge0$ (where $I$ is the $2\times 2$ identity matrix).
		\end{enumerate}
	\end{enumerate}
\end{hyp}

One can find examples of this model (or its analougue in the full space) in several references, e.g. \cite{DelFlaVin2014}, \cite{Ku90} and \cite{coghi_flandoli}.

\begin{ex}\label{example}
	We present here a family of $\sigma_k$ which satisfies Assumptions \ref{noise}.
	
	For every $k = (k_1, k_2) \in \mathbb{Z} ^2 \setminus\{0\}$ we define
	\begin{equation*}
	\sigma_k(x) = (\cos(k \cdot x) + \sin(k \cdot x))\; \frac{k^\perp}{\vert k \vert ^\beta}.
	\end{equation*}	
	Now we verify the Assumptions \ref{noise} for $\beta>4$.
	
	We have $\|\sigma_k\|_{C^h} \le C |k|^{-\beta+1+h}$, hence assumption $i)$ is satisfied for $\beta>4$. The Jacobian matrix is 
	\begin{equation*}
	D\sigma_k (x) = \frac{1}{\vert k \vert^\beta} (\cos(k \cdot x) - \sin(k \cdot x)) \left(\begin{array}{cc}
	-k_1 k_2 & -k_2 ^ 2\\
	k_1 ^ 2& k_1 k_2
	\end{array} \right).
	\end{equation*}
	The trace of this matrix is equal to $0$ and so assumption $ii)$ is also satisfied.
	
	The covariance matrix $Q$ is equal to
	\begin{equation*}
	Q(x, y) = \sum_{k \in \mathbb{Z} ^ 2 \setminus\{0\}} \frac{1}{\vert k \vert ^{2\beta}}\big[\cos(k \cdot(x - y)) + \sin(k \cdot (x + y))\big]\left(\begin{array}{cc}
	k_2 ^ 2 & - k_2 k_1\\
	- k_2 k_1& k_1 ^ 2
	\end{array} \right).
	\end{equation*}
	Now we group together the terms with $k$ and $-k$: $\sin(k \cdot (x + y))$ disappears and we get (calling $\mathbb{Z}^2_+= \mathbb{Z}_+ \times \mathbb{Z} \cup \{0\} \times \mathbb{Z}_+$)
	\begin{equation*}
	Q(x, y) = 2 \sum_{k \in \mathbb{Z}^2_+} \frac{1}{\vert k \vert ^{2\beta}}\cos(k \cdot(x - y))\left(\begin{array}{cc}
	k_2 ^ 2 & - k_2 k_1\\
	- k_2 k_1& k_1 ^ 2
	\end{array} \right).
	\end{equation*}
	Thus $Q$ depends only on the difference $x-y$ and assumption $iii) - a)$ is satisfied.
	To verify assumption $iii) - b)$ we look at $Q(0)$: here the terms with $k$ and $k ^ \perp$ sum up to a diagonal matrix, precisely (calling $\mathbb{Z}^2_{++} = \{k\in\mathbb{Z}\mid k_1\ge 0,k_2>0\}$)
	\begin{align*}
	Q(0) &= 2 \sum_{k \in \mathbb{Z}^2_+} \frac{1}{\vert k \vert ^{2\beta}}\left(\begin{array}{cc}
	k_2 ^ 2 & - k_2 k_1\\
	- k_2 k_1& k_1 ^ 2
	\end{array} \right)
	= 2 \sum_{k \in \mathbb{Z}^2_{++}} \frac{1}{\vert k \vert ^{2\beta}}\left(\begin{array}{cc}
	1 & 0\\
	0 & 1
	\end{array} \right).
	\end{align*}
	This shows that $Q(0) = a I$ for some $a$, that is assumption $iii) -b)$.
\end{ex}

\subsection{The Biot-Savart kernel}

We recall here the needed properties of the $2$-dimensional Biot-Savart kernel. The following results are standard for the Green function and can be found, among others, in \cite{Marchioro_Pulvirenti} and \cite{Brzezniak_Flandoli_Maurelli}.

For an $r> 0$, we define
\begin{equation}\label{gamma_log}
\gamma(r) = \left\{ 
\begin{array}{ll}
r(1- \operatorname{log}(r))  & \mbox{if } 0<r<1/e\\
r+1/e& \mbox{if } r\geq 1/e.
\end{array}
\right.
\end{equation}

For this function $\gamma$ and the Biot-Savart kernel $K$, the following properties hold:
\begin{enumerate}[label = (\roman*)]
	\item for every $0 < \e < \frac1e$,
	\begin{equation*}
	\gamma(r) \leq -\log(\e)r + \e.
	\end{equation*}
	\item $K$ is a divergence free vector field and
	\begin{equation*}
	\int_{\ \mathbb{T}^2 } \lvert K(x - y) - K(x^\prime - y)\rvert dy \lesssim \gamma(\lvert x - x^\prime \rvert), \quad \forall x,x^\prime \in \ \mathbb{T}^2 .
	\end{equation*}
	\item for every $\xi \in L^\infty$,
	\begin{equation*}
	\lvert (K\ast \xi_t)(x) - (K\ast \xi_t)(x^\prime)\rvert \lesssim \lVert \xi_t \rVert_{L^\infty} \gamma(\lvert x - x^\prime \rvert), \quad \forall x,x^\prime \in \ \mathbb{T}^2 .
	\end{equation*}
\end{enumerate}

\section{Vortex approximation for regularized Euler equations}\label{Regular_Section}

In this section, we work with a regularized kernel and we show the convergence of the particle system to the regularized Euler equation (Corollary \ref{cor: convergence regular case}).
The idea is the following. First, we show the existence and uniqueness for the regularized Euler equation by expressing any solution as fixed point of a certain operator $\Psi$ on the space $V_M^{1,T}$ and proving the contraction property for this operator. Then we note that both the weighted empirical measure $S_t^{N, \e}$ and the desired limit $\xi_t^{\e}$ are solutions to the regularized Euler equation, with different initial data and, with the previous representation in mind, we prove the convergence theorem as a continuity theorem with respect to the initial data.

For $\e  >0$ we take the mollifier $\rho^\e(x) = \e^{-2}\rho(\e^{-1}x) $, for $x \in \R^2$, where $\rho \in C^{\infty}_0(\R^2)$, $\rho \geq 0$, $\rho (-x) = \rho(x)$ and $\| \rho \|_{L^1} = 1$. We define 
\begin{equation*}
K^{\e} (x) := \int_{\R^2} K(x-y) \rho^{\e} (y) dy,
\qquad
x \in \T^2.
\end{equation*}
This function has the following properties.
\begin{lemma}
	\label{lem: properties mollified kernel}
	Let $\e  > 0$, the following holds:
	\begin{enumerate}[label=(\roman*), ref=(\roman*)]
		\item 
		\label{lem: properties mollified kernel: L1}
		$\Vert K^\e  - K \Vert_{L^1(\T^2)} \to 0$, as $\e \to 0$.
		
		\item 
		\label{lem: properties mollified kernel: C infty}
		$K^\e   \in C^\infty_b (\T^2; \R^2)$.
		
		\item 
		\label{lem: properties mollified kernel: explosion}
		For any $\delta > 0$ and $k \in \N$ there exists $C = C(\rho, k, \delta) > 0$ such that
		\begin{equation*}
		\| DK^\e \|_{C^k} \leq C \| K \|_{L^{2 / ( 1 + \delta ) } } \e ^ { - ( k+1+\delta ) }.
		\end{equation*}
	\end{enumerate}
\end{lemma}

\begin{proof}
	\ref{lem: properties mollified kernel: L1} and \ref{lem: properties mollified kernel: C infty} are standard properties of the mollification, we only show \ref{lem: properties mollified kernel: explosion}.
	Using H\"older inequality, with $q = 2/(1-\delta)$, and the change of variables $y^{\prime} = \e^{-1}y$, we get
	\begin{align*}
	| D^k K^{\e}(x) |
	\leq & \| K \|_{L^{2 / ( 1 + \delta ) } } \| D^k \rho^{\e} \|_{L^{q}}
	= \| K \|_{L^{2 / ( 1 + \delta ) } } \left ( \int_{\R^2} | \e^{-2}\e^{-k} D^k \rho (\e^{-1} y) | ^ q dy \right) ^{\frac{1}{q}}\\
	\leq & \| K \|_{L^{2 / ( 1 + \delta ) } } \e^{ (- (2+k)q + 2) /q}  \left ( \int_{\R^2} |  D^k (y) | ^ q \rho dy \right) ^{\frac{1}{q}}.
	\end{align*}
	This concludes the proof.
\end{proof}
Having the regularized kernel, for every $ \mu \in \mathcal{M}( \T^2 ) $ we define (omitting the $\e$ dependence in the notation)
\begin{equation*}
b ( x , \mu ) := \int_{\T^2} K^{\e }( x - y ) \mu ( dy ),
\qquad
x \in \T^2.
\end{equation*}

\begin{remark}\label{lipschitz b}
	The function $b$ is locally uniformly Lipschitz continuous in both arguments, precisely: 
	\begin{itemize}
		\item For every $x, x^{\prime} \in \T^2$ and $\mu\in \mathcal{M}(\T^2)$,
$
		| b(x,\mu) - b(x^\prime, \mu) | \leq \mbox{Lip}(K^\e)\Vert\mu\Vert \vert x-x^{\prime}\vert.
$
		\item For every$x \in \T^2$ and $ \mu, \mu^{\prime} \in \mathcal{M}(\T^2)$,
$
		| b (x,\mu) - b (x, \mu^{\prime}) | \leq \mbox{Lip}(K^\e)W_1(\mu, \mu^{\prime}).
$
	\end{itemize}
\end{remark}
We introduce the regularized Euler equation in vorticity form, which takes the form
	\begin{equation}
	\label{regular euler}
	\partial_t \mu + \mbox{div}(b(\mu) \mu) + \sum_{k=1} \mbox{div}(\sigma_k \mu) \circ d W^k_t = 0.
	\end{equation}
	For the rigorous definition, we consider the It\^o formulation of the above equation. Note that, for the assumptions on the noise (see e.g. \cite[Section 2.2]{coghi_flandoli}), the It\^o formulation reads formally:
	\begin{equation*}
	\partial_t \mu + \mbox{div}(b(\mu) \mu) + \sum_{k=1} \mbox{div}(\sigma_k \mu) d W^k_t = \frac12 \Delta \mu.
	\end{equation*}
We study distributional solutions of this equation in the following sense.
\begin{definition}
	\label{def: sol regular pde}
	Let $\mu_0 \in \mathcal{M}_M(\T^2)$. We say that $\mu \in V_M^{p, T}$ is a solution to equation \eqref{regular euler}  if, for every $\varphi \in C^2(\T^2)$,
	\begin{equation}
	\label{eq: weak formulation regular pde}
	\mu_s(\varphi) = \mu_0(\varphi) 
	+ \int_{0}^{t} \left[ \mu_s( b( \mu_s ) \nabla \varphi) + \frac{1}{2} \mu_s( \Delta \varphi ) \right] ds
	+ \sum_{k \geq 1} \int_{0}^{t}  \mu_s ( \sigma_k \nabla \varphi ) d W^k_s,
	\qquad
	\mathbb{P}-a.s.
	\quad
	\forall t\in[0,T]
	\end{equation}
	(the $\mathbb{P}$-exceptional set being independent of $t$).
\end{definition}	
\begin{remark}
	The stochastic integral in \eqref{eq: weak formulation regular pde} is well-defined, because $\mathbb{P}-$a.s.
	\begin{equation*}
	\sum_{k\geq 1} \mu_t ( \sigma_k \nabla \varphi ) ^2  \le  M \| \varphi \|_{C^2}^2 \sup_{x\in \T^2} \sum_{k \geq 1} | \sigma_k(x) | < \infty,\quad \forall t \in [0,T].
	\end{equation*}
\end{remark}
	
To handle the nonlinearity in equation \eqref{regular euler}, we fix a positive constant $M > 0$ and define the following auxiliary stochastic differential equation, where the random measure $\mu\in V_M^{p, T}$ is fixed:
\begin{equation}\label{auxiliary equation}
\left\{\begin{array}{l}
dX_t = b (X_t,\mu_t)dt+\sum_{k\geq1}\sigma_k(X_t)dW_t^k, \\
X_0 = x\in\T^2.
\end{array}\right.
\end{equation}

\begin{remark}\label{lipschitz flow}
	Since the coefficients are Lipschitz continuous and bounded, equation \eqref{auxiliary equation} admits a unique strong solution for $t\in[0,T]$. Moreover, there exists a version of the solution map $\Phi(t,x,\omega)$ which is Lipschitz continuous in the initial datum $x$, uniformly in $t$, and H\"older continuous in $t$, uniformly in $x$ (see \cite[Theorem 4.6.5]{Ku90}).
\end{remark}
We call $\Phi^{\mu}(t,x)$ the flow associated with equation \eqref{auxiliary equation}, to stress the dependence on $\mu$ of the drift.
For every measure $\mu_0 \in \mathcal{M}( \mathbb{T}^2) $, with $\Vert\mu_0\Vert \leq M$, we define the operator
\begin{equation*}\begin{array}{rlcl}
\Psi^{\mu_0}: & V_M^{p,T} & \to & V^{p,T}_M \\
& \mu & \mapsto & \Phi^{\mu}_{\#} \mu_0.
\end{array}
\end{equation*}
Note that the map $t\mapsto (\Phi^{\mu}_t)_{\#} \mu_0$ is a.s. continuous, indeed, by Remarks \ref{Wasserstein_bound} and \ref{lipschitz flow}, for every $s<t$,
\begin{equation*}
W_1((\Phi^{\mu}_t)_{\#} \mu_0,(\Phi^{\mu}_s)_{\#} \mu_0)\le M \| \Phi^{\mu}_t - \Phi^{\mu}_s \|_\infty \le C_\omega M |t-s|^\alpha.
\end{equation*}
Moreover the total variation norm satisfies $\Vert \Psi^{\mu_0} \Vert \leq \Vert \mu_0 \Vert$, $\mathbb{P}	$-a.s., for every $t \in [0,T]$. Hence the operator $\Psi^{\mu_0}$ is well-defined.

Now we show that the operator is a contraction in the norm $d_p^c$, for a suitable $c$.
%

\begin{lemma}\label{contraction lemma}
	Let $T>0$ and $p>2$ be fixed. Assume $\| \mu_0 \| \leq M$. There exists a constant $c>0$, depending on $\e$ (and on $a, T, p, M$) such that 
	\begin{equation*}
	d^c_p(\Phi^{\mu}_{\#} \mu_0, \Phi^{\mu^\prime}_{\#} \mu_0) \leq \frac12 d^c_p(\mu,\mu^{\prime}),
	\end{equation*}
	for every $\mu$, $\mu^\prime$ in $V_M^{p,T}$. Moreover,
		\begin{align}
		\label{eq: growth of c}
		c
		= c ( \sigma, p, M, \e  )
		\sim \| DK^{\e} \| _{C^0} ^ {p^2/(2(p-2))},
		\qquad
		\mbox{as } \e  \to 0.
		\end{align}
\end{lemma}

\begin{remark}\label{rmk:dist_c}
	Here we see the main reason to introduce the distance $d_p^c$ in \eqref{eq:dist_c} with the $e^{-ct}$ factor: by a suitable choice of $c$, the map $\Psi^{\mu_0}$ is a contraction on $V^{p,T}_M$ with the distance $d_p^c$, without any need to take $T$ small. Beside being technically convenient, this choice allows also to avoid the double exponential rate in \cite[Theorem 4.1]{MarPul1982}.
\end{remark}

\begin{proof}[Proof of Lemma \ref{contraction lemma}]
	We estimate the difference of the two images in terms of the differences of the two flows, namely
	\begin{align}
	d^c_p(\Phi^{\mu}_{\#} \mu_0, \Phi^{\mu^{\prime}}_{\#} \mu_0) ^ p 
	\leq & \Vert\mu_0\Vert^p \sup_{x\in \mathbb{T}^2} \mathbb{E}\left[  \sup_{t\in[0,T]} e^{-pct}\vert \Phi_t^{\mu}(x) - \Phi_t^{\mu^{\prime}}(x) \vert ^ p \right]. \label{wasserstein_inizio_lemma}
	\end{align}
	For two $\mu, \mu^{\prime} \in V^{p, T}_M$ and $x \in \T^2$, we apply It\^o formula to $f_\eta(x) = (\vert x \vert^{2} + \eta)^{\frac{q}{2}}$. Here $q = \frac{p}{2}$ and we choose $\eta = 0$ if $q \geq 2$, $\eta > 0 $ if $1 < q < 2$. For this choice of $\eta$ and $q$, it holds
	\begin{equation}
	\label{eq: properties of f}
	\vert \nabla f_\eta(x) \vert \leq q \vert x \vert^{q - 1} 
	\quad \vert D^2 f_\eta (x) \vert \leq q(q-1) \vert x \vert^{q - 2}.
	\end{equation}
	Notice that we endow the space of matrices with Hilbert-Schmidt norm.
	
	Let $\bar c = \frac{pc}{2}$. Using It\^o Formula and Assumption \ref{noise} we obtain that for each $x \in \T^2$, it holds $\mathbb{P}$-a.s.: for every $t$,
	\begin{align}
	e^{-\bar ct}  f_\eta (\Phi^{\mu}_t(x) - \Phi^{\mu^{\prime}}_t(x) ) 
	\leq & (C(q, \sigma)^2 - \bar c)\int_{0}^{t}e^{-\bar cs}\vert \Phi^{\mu}_s(x) - \Phi^{\mu^{\prime}}_s(x) \vert^{q}\mathrm{d}s \nonumber \\
	& + q\int_{0}^{t} e^{-\bar cs}\vert \Phi^{\mu}_s(x) - \Phi^{\mu^{\prime}}_s(x) \vert^{q-1} \vert b_s(\Phi^{\mu}_s(x), \mu) - b_s(\Phi^{\mu^{\prime}}_s(x), \mu^{\prime}) \vert \mathrm d s \label{first of flow}\\
	& + \left\vert \sum_{k \geq 1} \int_{0}^{t}e^{-\bar cs} \nabla f_\eta ( \Phi^{\mu}_s(x) - \Phi^{\mu^{\prime}}_s(x) )  (\sigma_k(\Phi^{\mu}_s(x)) - \sigma_k(\Phi^{\mu^{\prime}}_s(x))) dW_s^k \right\vert, \label{second of flow}
	\end{align}
	where $C(q, \sigma)$ is a positive constant (depending on $q$ and $(\sigma_k)_k$ and possibly changing from one line to another).
	
	To estimate term \eqref{first of flow}, we use a triangular inequality and the Lipschitz property of $b(x,\mu)$, both in $x$ and $\mu$ (remember $\| \mu_t \| \leq M$).
	Hence, term \eqref{first of flow} is bounded by the following 
	\begin{equation*}
	q\mbox{Lip}^q(K^\e) 
	\left[ 
	M \int_{0}^{t} e^{-\bar cs} \vert \Phi^{\mu}_s(x) - \Phi^{\mu^{\prime}}_s(x) \vert^{q} \mathrm d s 
	+  \int_{0}^{t} e^{- \bar cs}\vert \Phi^{\mu}_s(x) - \Phi^{\mu^{\prime}}_s(x) \vert^{q-1}W_1(\mu_s,\mu_s^\prime) \mathrm d s 
	\right].
	\end{equation*}
	We apply Young inequality with $\frac{q}{q-1}$ and $q$ to the second term to obtain, for every $s \in [0,T]$ and $\delta >0$ (to be determined later),  
	\begin{equation*}
	\vert \Phi^{\mu}_t(x) - \Phi^{\mu^{\prime}}_t(x) \vert^{q-1}W_1(\mu_s,\mu_s^\prime) \leq \delta^{-1/(q-1)} \vert \Phi^{\mu}_t(x) - \Phi^{\mu^{\prime}}_t(x) \vert^{q}  + \frac{\delta}{q}  W_1(\mu_s,\mu_s^\prime)^{q}.
	\end{equation*}
	Substituting into \eqref{first of flow} we obtain
	\begin{align}
	\label{to_square}
	e^{-\bar ct} f_\eta (\Phi^{\mu}_t(x) - \Phi^{\mu^{\prime}}_t(x) ) 
	\leq & L \int_{0}^{t}e^{- \bar cs}\vert \Phi^{\mu}_s(x) - \Phi^{\mu^{\prime}}_s(x) \vert^{q}\mathrm{d}s
	 + \delta \mbox{Lip}^q(K^\e) \int_{0}^{t} e^{- \bar cs}W_1(\mu_s, \mu_s^\prime )^q \mathrm d s 
	+ \eqref{second of flow},
	\end{align}
	where $L = C(q, \sigma) + q\mbox{Lip}^q(K^\e) (M + \delta^{-1/(q-1)}) - \bar c$. We can choose $\bar c$ as a function of $\delta$, $M$, $Lip(K^\e)$ and $q$ such that
	\begin{equation*}
	L = C(q, \sigma) + q\mbox{Lip}^q(K^\e) (M + \delta^{-1/(q-1)}) - \bar c = 0,
	\end{equation*}
	so that we can remove the corresponding term from the estimates.
	We estimate now the expectation of the square of \eqref{second of flow}. First we use the Burkholder-Davis-Gundy inequality, then the Lipschitz assumption on $\sigma$ and \eqref{eq: properties of f} to obtain
	 \begin{align}
	 \mathbb{E} & \sup_{s\in[0,t]}\left\vert \sum_{k \geq 1} \int_{0}^{s}e^{-\bar cr} \nabla f_\eta ( \Phi^{\mu}_r(x) - \Phi^{\mu^{\prime}}_r(x) )  (\sigma_k(\Phi^{\mu}_r(x)) - \sigma_k(\Phi^{\mu^{\prime}}_r(x))) dW_r^k \right\vert^2 \nonumber\\
	 &\leq  C \mathbb{E} \sum_{k \geq 1} \int_{0}^{t} e^{-2\bar cr} |\nabla f_\eta ( \Phi^{\mu}_r(x) - \Phi^{\mu^{\prime}}_r(x) )|^2  |\sigma_k(\Phi^{\mu}_r(x)) - \sigma_k(\Phi^{\mu^{\prime}}_r(x))|^2 dr \nonumber\\
	 &\leq  C(q, \sigma) \mathbb{E} \int_{0}^t \sup_{r\in[0,s]} e^{- 2 \bar cr} \vert\Phi^{\mu}_r(x) - \Phi^{\mu^{\prime}}_r(x) \vert^{2q} \mathrm d s.\label{eq: martingale estimate}
	 \end{align}
Now we estimate the expectation of the square of \eqref{to_square}: using \eqref{eq: martingale estimate} and Jensen inequality we have (remember that $f_\eta(x) \geq \vert x \vert^q$ and $L=0$)
		\begin{align*}
		\mathbb{E} \sup_{s\in[0,t]} e^{-2\bar cs} \vert \Phi^{\mu}_s(x) - \Phi^{\mu^{\prime}}_s(x) \vert^{2q}
		\leq \; & T\delta^2 \mbox{Lip}^{2q}(K^\e) \int_{0}^{t} e^{- 2\bar cs}W_1(\mu_s, \mu_s^\prime )^{2q} \mathrm d s\\
		& + C(q, \sigma) \mathbb{E} \int_{0}^t \sup_{r\in[0,s]} e^{- 2 \bar cr} \vert\Phi^{\mu}_r(x) - \Phi^{\mu^{\prime}}_r(x) \vert^{2q} \mathrm d s.
		\end{align*}
Applying Gronwall's Lemma, we obtain (remember  that $2q = p$ and $2\bar c = pc$)
	\begin{equation*}
	\mathbb{E} \sup_{s\in[0,t]} e^{-pcs} \vert \Phi^{\mu}_s(x) - \Phi^{\mu^{\prime}}_s(x) \vert^p 
	\leq T\delta^{2} \mbox{Lip}^p(K^\e) e^{C(q, \sigma) T}  \int_{0}^{t}\mathbb{E} \left[\sup_{r\in[0,s]} e^{-pcr}W _1(\mu_r,\mu_r^{\prime})^p \right] \mathrm d s. 
	\end{equation*}
	We can choose now $\delta = e^{-\frac12 C(q, \sigma)T }(2^pM^pT\operatorname{Lip}^p(K^\e))^{-\frac12} $. With this choice, we get
	\begin{align*}
	c
	= c ( \sigma, p, M, \e  )
	\sim \| DK^{\e} \| _{C^0} ^ {p/2} \| DK^{\e} \| _{C^0} ^ {p/2\cdot 2/(p-2)} = \| DK^{\e} \| _{C^0} ^ {p^2/(2(p-2))},
	\qquad
	\mbox{as } \e  \to 0.
	\end{align*}		
	With the constants chosen in this way we obtain
	\begin{equation}\label{expectation flow}
	\mathbb{E} \left[\sup_{t\in[0,T]} e^{-pct}\vert \Phi^{\mu}_t(x) - \Phi^{\mu^{\prime}}_t(x) \vert^p\right] 
	\leq \frac{ 1 }{ 2^p M^p }  \mathbb{E}\left[\sup_{t\in[0,T]} e^{-ct}W_1(\mu_t,\mu_t^{\prime})^p\right].
	\end{equation}
	Estimates \eqref{expectation flow} and \eqref{wasserstein_inizio_lemma} allow to conclude the proof of the lemma.
\end{proof}

\begin{theorem}\label{ex_weak_sol}
	Let $T, M > 0$ and $p > 2$. Let $\mu_0 \in \mathcal{M}_M(E)$
	Then \eqref{regular euler} has a solution in the space $V_M^{p,T}$ starting from $\mu_0$. Precisely, the unique fixed point for the operator $\Psi^{\mu_0}$ is a solution to \eqref{regular euler}.
\end{theorem}

\begin{proof}
	From Lemma \ref{contraction lemma} follows that there exists $c>0$ such that the operator $\Psi^{\mu_0}$ is a contraction $(V^{p,T}_M, d^c_p)$. Hence it has a unique fixed point. 
	As a straight forward application of It\^{o} formula one can show that every fixed point satisfies \eqref{eq: weak formulation regular pde}. The proof is complete.
\end{proof}

Now we investigate the continuity of the fixed point of $\Psi^{\mu_0}$ with respect to the initial condition $\mu_0$. We need a preliminary estimate on the derivative of the flow $\Phi$ associated to \eqref{auxiliary equation}.

\begin{lemma}\label{estimate derivative}
	Let $\Phi_t$ be the stochastic flow of equation \eqref{auxiliary equation}, we denote by $D\Phi_t(x)$ its derivative in space. For every every $p>3$, we have
	\begin{equation*}
	\mathbb{E}\left[ \sup_{t\in[0,T]}\sup_{x\in \mathbb{T}^2} \vert D\Phi_t(x) \vert ^p \right] \leq \Lambda( p),
	\end{equation*}
	where $\Lambda(p)= C \left( 1 +  M^{p} \Vert DK^\e\Vert_{C^1}^p \right) 
	\exp \left(
	C
	( 1 +  M^{2p} \Vert DK^\e\Vert_{C^0}^{2p} )
	\right)$ and $C = C(p, T,  \Vert \sigma \Vert_{C^2})$.
\end{lemma}

\begin{proof}
	Let $\eta_t(x) = D\Phi_t(x)$ be the space derivative of the flow. By formal computation $\eta$ satisfies the following stochastic differential equation, for every $x \in \T^2$,
	\begin{equation*}
	\eta_t(x) = I + \int_{0}^{t} Db(\Phi_u(x)) \eta_u(x)\mathrm{d}u + \int_{0}^{t}\sum_{k} D\sigma_k(\Phi_u(x)) \eta_u(x) \mathrm{d} W_u^k,
	\qquad \forall t\in[0,T],\quad
	\mathbb{P}-a.s.
	\end{equation*}
	where $I$ is the $2\times 2$ identity matrix.
	
	We estimate the $L^{\gamma}(\Omega)$ norm of $\eta$, for fixed $\gamma > 2$ and fixed $x \in \mathbb{T}^2$, $t\in [0,T]$. By a standard computation, which we do not repeat here, we obtain, for every $t\in[0,T]$,
	\begin{equation}\label{moment of flow}
	\mathbb{E}\left [ \left \vert \eta_t(x) \right\vert ^{\gamma} \right ] 
	\leq c_{\gamma}
	\exp \left(
	C
	( 1 +  M^{\gamma} \Vert DK^\e\Vert_{C^0}^{\gamma} )
	\right),
	\end{equation}
	where $c_\gamma$, $C$ are positive constants depending respectively on $\gamma$ and on $\gamma, T, \Vert D\sigma \Vert_{C^0}$.
	In view of Kolmogorov criterion, we would like to control, for fixed $x,y \in \mathbb{T}^2$, $t>s \in [0,T]$ and $\gamma \geq 2$,
	\begin{equation}
	\mathbb{E}\left \vert \eta_t(x) - \eta_s(y)\right\vert^\gamma \leq \mathbb{E}\left \vert \eta_s(x) - \eta_s(y)\right\vert^\gamma +\mathbb{E}\left \vert \eta_t(x) - \eta_s(x)\right\vert^\gamma.\label{Kolmogorov_split}
	\end{equation}	
	For the first addend in the right hand side of \eqref{Kolmogorov_split}, we have
	\begin{align}
	\mathbb{E}\left \vert \eta_s(x) - \eta_s(y)\right\vert^\gamma \leq & c_\gamma \mathbb{E}\left\vert \int_{0}^{s}\left( Db(\Phi_u(x))\eta_u(x) - Db(\Phi_u(y))\eta_u(y)\right) \mathrm d u \right \vert^\gamma \label{derivative drift}\\
	&+ c_{\gamma} \mathbb{E}\left\vert \int_{0}^{s} \sum_k \left( D\sigma_k(\Phi_u(x))\eta_u(x) - D\sigma_k(\Phi_s(y))\eta_u(y) \right) \mathrm d W_u ^k \right\vert^\gamma. \label{derivative vort}
	\end{align}	
	We first estimate term \eqref{derivative drift}. Using Cauchy-Schwarz inequality, we obtain for any $0<\alpha < 1$,
	\begin{align}
	c_\gamma \mathbb{E}&\left\vert \int_{0}^{s}\left( Db(\Phi_u(x))\eta_u(x) - Db(\Phi_u(y))\eta_u(y)\right) \mathrm d u \right \vert^\gamma \nonumber\\
	& \leq c_{\gamma} T^{\gamma - 1} \mathbb{E} \left[ \int_{0}^{s} \left( \vert Db(\Phi_u(x)) -Db(\Phi_u(y))\vert^\gamma \vert \eta_u(x)\vert^\gamma + \vert Db(\Phi_u(y)) \vert^\gamma \vert \eta_u(x) - \eta_u(y)\vert^\gamma  \right)\mathrm d u \right] \nonumber \\
	& \leq c_{\gamma} T^{\gamma-1} \Vert DK^\e\Vert_{C^\alpha}^\gamma M^\gamma \left( \int_{0}^{s} \mathbb{E}\left\vert \Phi_u(x) - \Phi_u(y) \right\vert^{2\alpha \gamma} \mathrm d u \right)^{\frac12} \left(\int_{0}^{s} \mathbb{E}\left\vert \eta_u(x) \right\vert^{2\gamma} \mathrm d u \right)^{\frac12}\label{derivative space diff first}\\
	& \quad + c_{\gamma} T^{\gamma-1}\Vert DK^\e\Vert_{C^0}^{\gamma} M^\gamma \int_0^s \mathbb{E} \left \vert \eta_u(x) - \eta_u(y)\right \vert ^\gamma \mathrm d u. \nonumber
	\end{align}	
	In a similar way, we can apply the Burkholder-Davis-Gundy inequality to \eqref{derivative vort}, then apply the same reasoning as before to obtain
	\begin{align}
	c_{\gamma} \mathbb{E}&\left\vert \int_{0}^{s} \sum_k \left( D\sigma_k(\Phi_u(x))\eta_u(x) - D\sigma_k(\Phi_s(y))\eta_u(y) \right) \mathrm d W_u^k \right\vert^\gamma \nonumber\\
	&\leq c_{\gamma} T^{\frac \gamma 2-1}\Vert D\sigma \Vert_{C^0}^{\gamma} \int_0^s \mathbb{E} \left \vert \eta_u(x) - \eta_u(y)\right \vert ^\gamma \mathrm d u \nonumber\\
	& \quad + c_{\gamma} T^{\frac \gamma 2-1}\Vert D\sigma \Vert_{C^\alpha}^{\gamma} \left( \int_{0}^{s} \mathbb{E}\left\vert \Phi_u(x) - \Phi_u(y) \right\vert^{2\alpha\gamma} \mathrm d u \right)^{\frac12} \left(\int_{0}^{s} \mathbb{E}\left\vert \eta_u(x) \right\vert^{2\gamma} \mathrm d u \right)^{\frac12}. \label{derivative space diff second}
	\end{align}
	In a standard way we estimate the difference (using $a^{2\alpha\gamma}\le 1+a^{(2\alpha\gamma)\vee 1}$:
	\begin{align}
	\mathbb{E}\left\vert \Phi_u(x) - \Phi_u(y) \right\vert^{2\alpha\gamma} &\le \int_0^1 (1+\mathbb{E}\left\vert \eta_u(\xi x +(1-\xi)y)\right\vert^{(2\alpha\gamma) \vee 1}) \mathrm d \xi |x-y|^{2\alpha\gamma} \nonumber\\
	&\le \left(1+\sup_z\mathbb{E}\left\vert \eta_u(z)\right\vert^{(2\alpha\gamma) \vee 1}\right) |x-y|^{2\alpha\gamma}.
	\label{Lip_expectation}
	\end{align}
	We put together now estimates \eqref{derivative space diff first}, \eqref{derivative space diff second}, \eqref{Lip_expectation} and \eqref{moment of flow} to get (using $a^\gamma \le 1+a^{2\gamma}$ and $a^{(2\alpha\gamma) \vee 1} \le 1+a^{2\gamma}$ for every $a\ge 0$, $0<\alpha\le 1$)
	\begin{align}
	\mathbb{E} \left \vert \eta_s(x) - \eta_s(x)\right\vert^\gamma \nonumber 
	 \leq  & C 
	( 1+  M^{\gamma} \Vert DK^\e\Vert_{C^\alpha}^\gamma )
	\exp \left(
	C
	( 1 +  M^{2 \gamma} \Vert DK^\e\Vert_{C^0}^{2 \gamma} )
	\right)
	|x-y|^{\alpha\gamma} \nonumber \\
	& + C ( 1+  M^{\gamma} \Vert DK^\e\Vert_{C^0}^\gamma ) 
	\int_{0}^{s}\mathbb{E}\vert \eta_u(x) - \eta_u(y)\vert^{\gamma} \mathrm d u .
	\label{stima derivata}
	\end{align}
	Here and in the rest of the proof $C = C(\gamma, T,  \Vert D\sigma \Vert_{C^\alpha},  \Vert D\sigma \Vert_{C^0})$ . By Gronwall's Lemma we obtain 
	\begin{equation}\label{stima finale derivate spazio}
	\mathbb{E}\left [\vert \eta_s(x) -\eta_s(y)\vert^{\gamma}\right] 
	\leq
	C 
	( 1+  M^{\gamma} \Vert DK^\e\Vert_{C^\alpha}^\gamma )
	\exp \left(
	C
	( 1 +  M^{2 \gamma} \Vert DK^\e\Vert_{C^0}^{2 \gamma} )
	\right)
	|x-y|^{\alpha\gamma}.
	\end{equation}
	For the second term in \eqref{Kolmogorov_split}, we have
	\begin{align}
	\mathbb{E}\left \vert \eta_t(x) - \eta_s(x)\right\vert^\gamma \leq
	& c_{\gamma} \mathbb{E} \left\vert \int_{s}^{t} Db(\Phi_u(x))\eta_u(x) \mathrm d u\right\vert^{\gamma} + c_{\gamma} \mathbb{E}\left\vert\int_{s}^{t} \sum_{k}D\sigma_k(\Phi_u(x))\eta_u(x)\mathrm d W_u^k \right\vert^\gamma. \label{derivative time}
	\end{align}	
	The two terms in \eqref{derivative time} can be estimated using the boundedness of the coefficients, Burkholder-Davis-Gundy inequality for the second one, H\"older inequality and  \eqref{moment of flow} to obtain
	\begin{align}
	\mathbb{E}&\left \vert \eta_t(x) - \eta_s(x)\right\vert^\gamma 
	\leq 
	C \left( 1 +  M^{\gamma} \Vert DK^\e\Vert_{C^0}^\gamma \right) 
	\vert t - s \vert ^{\frac{\gamma}{2}} 
	\sup_{u\in[0,T]}\mathbb{E}\vert \eta_u(x) \vert^\gamma \nonumber \\
	& \leq C \left( 1 +  M^{\gamma} \Vert DK^\e\Vert_{C^0}^\gamma \right) 
	\exp \left(
	C
	( 1 +  M^{\gamma} \Vert DK^\e\Vert_{C^0}^{\gamma} )
	\right)
	\vert t - s \vert ^{\frac{\gamma}{2}}.
		\label{derivative time diff}
	\end{align}
%
Finally, we use \eqref{stima finale derivate spazio} and \eqref{derivative time diff} and the inequality $\Vert DK^\e\Vert_{C^\alpha} \leq \Vert DK^\e\Vert_{C^1}$ to obtain
	\begin{equation}\label{stima finale derivate}
	\mathbb{E}\left [\vert \eta_t(x) -\eta_s(y)\vert^{\gamma}\right] 
	\leq \Lambda(\gamma)\left( \vert x - y \vert^{\alpha\gamma} + \vert t - s \vert^{\frac{\gamma}{2}} \right),
	\end{equation}
	where $\Lambda (\gamma)= C \left( 1 +  M^{\gamma} \Vert DK^\e\Vert_{C^1}^\gamma \right) 
	\exp \left(
	C
	( 1 +  M^{2\gamma} \Vert DK^\e\Vert_{C^0}^{2\gamma} )
	\right)
	$.
	In order to conclude, we recall the following inequality, a consequence of the Sobolev Embedding Theorem, valid for $\alpha'>0$, $\beta := \alpha^\prime - 3/p > 0 $:
	\begin{equation*}
	\mathbb{E}\left[ \sup_{t,s\in [0,T]}\sup_{x,y\in \mathbb{T}^2} \frac{\vert D\Phi_t(x)-D\Phi_s(y)\vert^p}{(\vert t - s \vert ^2 + \vert x - y \vert^2)^{\frac{\beta p}{2}}} \right] 
	\leq \iint_{[0,T] \times \mathbb{T}^2} \frac{\mathbb{E}\vert D\Phi_t(x) - D\Phi_s(y)\vert ^p }{\left( \vert t - s \vert^2 + \vert x - y \vert^2 \right)^{\frac32 + \frac{\alpha^\prime p}{2}}} \mathrm d t \mathrm d s \mathrm d x \mathrm d y.
	\end{equation*}
	Now we use \eqref{stima finale derivate} with $\gamma=p$ to obtain
	\begin{equation}
	\mathbb{E}\left[ \sup_{t,s\in [0,T]}\sup_{x,y\in \mathbb{T}^2} \frac{\vert D\Phi_t(x)-D\Phi_s(y)\vert^p}{(\vert t - s \vert ^2 + \vert x - y \vert^2)^{\frac{\beta p}{2}}} \right] 
	\leq \Lambda( p ) \iint_{[0,T] \times \mathbb{T}^2} \frac{\vert t - s \vert^{\frac p2} + \vert x - y \vert^{\alpha p}}{\left( \vert t - s \vert^2 + \vert x - y \vert^2 \right)^{\frac32+ \frac{\alpha^\prime p}{2}}} \mathrm d t \mathrm d s \mathrm d x \mathrm d y. \label{Sobolev_emb}
	\end{equation}
	The condition $p>3$ guarantees that we can find $\alpha$ and $\alpha^\prime$ in $(0,1)$ with $\alpha^\prime - 3/p>0$ and $\alpha p -\alpha^\prime p -3 >-3$ so that the integral in the right hand side of \eqref{Sobolev_emb} is finite. The proof is complete.
\end{proof}

\begin{lemma}\label{Continuity initial condition}
	Let $T > 0$ and $p > 2 $, let $c$ be given as in Lemma \ref{contraction lemma}. Given $\mu_0, \nu_0 \in \mathcal{M}_M(\T^2)$ there exists a positive constant $\Gamma = \Gamma (p, T, \sigma, M, \e ) := \Lambda ( p)^{\frac{1}{p}} \vee \Lambda(4)^ {\frac{1}{4}}$, such that
	\begin{equation*}
	d_p^c(\mu, \nu) \leq 2\Gamma W_1(\mu_0, \nu_0),
	\end{equation*}
	where $\mu, \nu \in V$ are the fixed points of operators $\Psi^{\mu_0}, \Psi^{\nu_0}$ respectively.
\end{lemma}

\begin{proof}
	We use a triangular inequality to get
	\begin{equation}
	\label{stima_lemma_2}
	d_p^c (\mu, \nu) =  \; d_p^c( \Phi^{\mu}_\# \mu_0, \Phi^{\nu}_\# \nu_0) 
	\leq  \; d_p^c( \Phi^{\mu}_\# \mu_0, \Phi^{\mu}_\# \nu_0) + d_p^c( \Phi^{\mu}_\# \nu_0, \Phi^{\nu}_\# \nu_0). 
	\end{equation}
	It follows from Lemma \ref{contraction lemma} that the second term in the right hand side is less than or equal to $\frac12 d_p^c(\mu, \nu)$. We look at the first term, which is, by definition,
	\begin{equation*}
	d_p^c( \Phi^{\mu}_\# \mu_0, \Phi^{\mu}_\# \nu_0) ^ p = \mathbb{E}\sup_{t\in[0,T]}e^{-ct}\left\vert \sup_{\varphi \in Lip_1(\T^2)} \left( \int \varphi \circ \Phi^{\mu} (x) \; (d \mu_0 - d \nu_0) (x) \right) \right\vert^p.
	\end{equation*}
	It follows from Remark \ref{lipschitz flow} that the flow $\Phi^{\mu}$ is a Lipschitz function on $\T^2$. Hence, for any Lipschitz function $\varphi$, also the function $\varphi \circ \Phi^{\mu}$ is Lipschitz. Hence we have
	\begin{equation*}
	d_p( \Phi^{\mu}_\# \mu_0, \Phi^{\mu}_\# \nu_0) \leq  \mathbb{E}[\sup_{t\in[0,T]}e^{-ct} \vert\mbox{Lip}(\Phi^{\mu})\vert^p ] ^ {\frac1p} W_1(\mu_0, \nu_0) .
	\end{equation*}
	We recall that
	\begin{equation*}
	e^{-ct} \vert \mbox{Lip}(\Phi_t^{\mu}) \vert \leq \sup_{x\in \T^2} \vert D\Phi^\mu_t(x)\vert,
	\end{equation*}
	where we used that $ e^{-ct} < 1$.
	To estimate this last term we use Lemma \ref{estimate derivative}: for any $p>3$ we have
	\begin{equation*}
	\mathbb{E} \left[ \sup_{t\in[0,T]}\sup_{x\in \mathbb{T}^2} \vert D\Phi_t(x) \vert ^p \right] \leq \Lambda(p).
	\end{equation*}
	If $2< p \leq 3$, we have
	\begin{equation*}
	\mathbb{E}\left[ \sup_{t\in[0,T]}\sup_{x\in \mathbb{T}^2} \vert D\Phi_t(x) \vert ^p \right] \leq \Lambda(4)^{\frac{p}{4}}.
	\end{equation*}
	Using this, we find that
	\begin{equation*}
	d_p^c( \Phi^{\mu}_\# \mu_0, \Phi^{\mu}_\# \nu_0) 
	\leq  \mathbb{E}\left[\sup_{t\in[0,T]}e^{-ct}\mbox{Lip}(\Phi^{\mu})^p  \right] ^ {\frac1p} W_1(\mu_0, \nu_0)
	\leq \Gamma W_1(\mu_0, \nu_0).
	\end{equation*}
	Putting together the estimates on the two terms of \eqref{stima_lemma_2} we conclude the proof.
\end{proof}

Lemma \ref{Continuity initial condition} states the continuity of the fixed point of $\Psi^{\mu_0}$ with respect to the initial condition $\mu_0$. We use this to study the mean-field convergence of the particle system \eqref{eq:part_reg_intro}. We recall that the particle system reads, for $N\in\N$ and $1\leq i \leq N$,
\begin{equation}
\label{eq: regular particles}
d X_t^{i,N}=\frac{1}{N}\sum_{j=1}^{N}\xi^{ j , N}K^{\e}(X_t^{i,N}-X_t^{j,N})dt+\sum_{k=1}^{\infty} \sigma_k(X_t^{i,N})\circ dW_t^k,
\qquad
X_t^{ i , N } |_{ t = 0 } = x^{i, N}.
\end{equation}
Here $(x^{i,N})_{1\leq i \leq N} \subseteq \T^2$ and $(\xi^{i,N})_{1\leq i \leq N}\subseteq \mathbb{R}$ are given.

\begin{corollary}
	\label{cor: convergence regular case}
	Let $T, M>0$ and $p > 2$. Let $\mu$ be a solution to equation \eqref{regular euler} and $S^{N, \e}_t = \frac{1}{N} \sum_{i=1}^{N} \xi^{i,N} \delta_{X_t^{i,N}} $ the empirical measure associated to the system of particles \eqref{eq: regular particles} with $\|\mu_0\| \vee \|S^{N,\e}_0\| \le M$. Then,
	\begin{equation*}
	d_p(S^N, \mu) 
	\leq e^{cT} d_p^c(S^N, \mu) 
	\leq 2\Gamma e^{cT} W_1(S^N_0,\mu_0),
	\end{equation*}
	where $c$ is given in \eqref{eq: growth of c} and $\Gamma$ is defined in Lemma \ref{Continuity initial condition}.
\end{corollary}

\begin{proof}
	We show now that for every $N \in \N$ the empirical measure $S^{\e, N}$ associated to the system of interacting particles \eqref{eq: regular particles} driven by $K^\e $ is indeed a fixed point for the operator $\Psi^{S^N_0}$. We must prove
	\begin{equation*}
	S_t^{\e,N} = (\Phi_t^{S^{\e,N}})_\sharp S_0^N,
	\qquad
	t \in [0,T].
	\end{equation*}
	We evaluate the right hand side in a test function $\varphi \in C(\T^2; \R^2)$,
	\begin{equation}\label{empirical flow}
	(\Phi^{S_t^{\e,N}}_\sharp S_0^N)(\varphi) = \sum_{i=1}^N \varphi(\Phi^{S_t^{\e,N}}(x^{i,N})).
	\end{equation}
	Since, by definition, $\Phi^{S_T^{\e,N}}$ is the flow associated with the equation \eqref{auxiliary equation} with drift depending on the empirical measure, it is immediate to see that $\Phi^{S_t^{\e,N}}(x^{i,N}) = X_t^{i,N}$. Thus, the right hand side of \eqref{empirical flow} is exactly $S_t^{\e, N}(\varphi)$.
	
	Now, since both $\mu$ and the empirical measure $S^{\e,N}$ are solutions in $V^{p,T}_M$ to the limit equation \eqref{regular euler}, given as a fixed point of the map $\Psi$, we conclude using Lemma \ref{Continuity initial condition}.
\end{proof}

\section{Convergence of regularized Euler equations}\label{Euler_Section}

In this section we show the convergence of the regularized Euler equation to the (true) Euler equation.

For a given initial condition $\xi_0\in L^\infty(\ \mathbb{T}^2 )$, we consider the flow associated with the approximated kernel $K^\e $, namely
\begin{equation*}
d\Phi^\e(x) = \int_{\ \mathbb{T}^2 }K^\e(\Phi^\e(x)-\Phi^\e(y))\xi_0(y)\ud y dt + \sum^\infty_{k=1}\sigma_k(\Phi^\e(x))dW^k.
\end{equation*}
The existence and uniqueness of $\Phi^\e$ follows from the previous section. We also consider the flow $\Phi$ associated with the true Euler equation, namely
\begin{equation*}
d\Phi(x) = \int_{\ \mathbb{T}^2 }K(\Phi(x)-\Phi(y))\xi_0(y)\ud y dt + \sum^\infty_{k=1}\sigma_k(\Phi(x))dW^k.
\end{equation*}
In \cite{Brzezniak_Flandoli_Maurelli} the authors show existence and uniqueness for $\Phi$ and proves that the measure $\xi_t := (\Phi_t)_{\sharp}\xi_0$
is a solution to the stochastic Euler vorticity equation \eqref{eq:vort_intro}. The following result shows the convergence of $\Phi^\e$ to $\Phi$. The result is adapted from \cite{Brzezniak_Flandoli_Maurelli}; the main improvement here is to bring the supremum over time inside the expectation and take the $L^p$ norm in $\omega$.

\begin{theorem}\label{convergence_approximation}
	For every $p\ge1$ finite, the family $(\Phi^\e)_\e$ converges to $\Phi$ in $L^p_\omega(C_t(L^1_x))$ (as $\e\rightarrow0$). Moreover it holds, for some positive constants $C$ (depending on $p$, $T$, $\|\xi_0\|_{L^\infty_x}$ and $\sum_k\|\sigma_k\|_{W^{1,\infty}_x}^2$) and $c$ (depending on $p$ and $T$)
	\begin{equation*}
	\E\left( \sup_{t\in[0,T]} \int_{\ \mathbb{T}^2 } |\Phi^\e_t(x)-\Phi_t(x)| \ud x\right)^p \le C \|K^\e-K\|_{L^1_x}^{p\exp(-c\|\xi_0\|_{L^\infty_x}t)}.
	\end{equation*}
\end{theorem}

\begin{proof}
	In the proof, unless otherwise stated, $C,C',c,\ldots $ denote constants that can depend on $p$ and $T$. Call $Z^\e_t(x)=\Phi^\e_t(x)-\Phi_t(x)$. We expect, from the deterministic theory (see e.g.\ \cite{Marchioro_Pulvirenti}) and the stochastic counterpart (see \cite{Brzezniak_Flandoli_Maurelli}), that $\E\sup_{s\in[0,t]}\|Z^\e_{s}\|_{L^1_x}^p$ satisfies a differential inequality involving a log-Lipschitz drift, and therefore we expect to get an estimate of the form $\E\sup_{ s \in[0,t]}\|Z^\e_{ s }\|_{L^1_x}^p \le \|K_\e  -K\|^{pe^{-\lambda t}}$. The problem, with respect to \cite{Brzezniak_Flandoli_Maurelli}, comes from the supremum over time inside the expectation, which does not allow easily a comparison principle. For this reason, we choose not to control directly the $L^p_\omega(C_t(L^1_x))$ norm of $Z$, but rather
	\begin{align*}
	\E\sup_{ s \in[0,t]}\|Z^\e_{ s }\|_{L^1_x}^{p(t)},
	\end{align*}
	where $p(t)=pe^{\lambda t}$ for some $\lambda\ge 0$ to be determined later and for $p\ge 2$. A bound on this quantity will imply the final estimate.
	
	As first step we compute the SDE for $\|Z^\e_t\|_{L^1_x}$. We would like to apply It\^o formula to $f(z)=|z|$, since this function is not $C^2$ we use the approximate function $f_\delta(z)=(|z|^2+\delta)^{1/2}$, $\delta>0$; we recall that $|\nabla f_\delta(z)|\le 1$ and $|D^2f_\delta(z)|\le |z|^{-1}$. Applying It\^o formula to $f_\delta(Z^\e)$ we get
	\begin{align*}
	df_\delta(Z^\e) = & \nabla f_\delta(Z^\e)(u^\e(\Phi^\e)-u(\Phi))dt
	 + \sum_k\nabla f_\delta(Z^\e)(\sigma_k(\Phi^\e)-\sigma_k(\Phi))dW^k\\
	& + \frac12\sum_k(\sigma_k(\Phi^\e)-\sigma_k(\Phi))\cdot D^2f_\delta(Z^\e)(\sigma_k(\Phi^\e)-\sigma_k(\Phi))dt.
	\end{align*}
	In the following, we use the notation $H_t = \int_{\ \mathbb{T}^2 }|Z^\e_t| \ud x$ and $H(\delta)_t = \int_{\ \mathbb{T}^2 }f_\delta(Z^\e_t) \ud x$. In order to estimate $H(\delta)$, we integrate in space and exchange integrals in space and in time using Fubini theorem and stochastic Fubini theorem: it holds
	\begin{align*}
	dH(\delta) = &\int_{\ \mathbb{T}^2 } \nabla f^\delta(Z^\e)(u^\e(\Phi^\e)-u(\Phi)) \ud x\ud t
	+\sum_k\int_{ \mathbb{T}^2 } \nabla f_\delta(Z^\e)(\sigma_k(\Phi^\e)-\sigma_k(\Phi)) \ud x\ud W^k\\
	& +\frac12\sum_k\int_{ \mathbb{T}^2 } (\sigma_k(\Phi^\e)-\sigma_k(\Phi))\cdot D^2f_\delta(Z^\e)(\sigma_k(\Phi^\e)-\sigma_k(\Phi))\ud x\ud t.
	\end{align*}
	To control $\|Z^\e_{ s }\|_{L^1_x}^{p(t)}$, we apply again It\^o formula to $H(\delta)^{p(t)/2}=\exp[p(t)\log H(\delta)/2]$ (the $p(t)/2$-power can be regarded as regular since $H(\delta)\ge \delta^{1/2}|\mathbb{T}^2|>0$): we get
	\begin{align*}
	&H(\delta)_t^{p(t)/2} -H(\delta)_0^{p/2} = \int^t_0 \frac{p(r)}{2}H(\delta)^{p(r)/2-1} \int_{\ \mathbb{T}^2 } \nabla f^\delta(Z^\e)(u^\e(\Phi^\e)-u(\Phi)) \ud x\ud r\\
	&\ \ + \sum_k\int^t_0 \frac{p(r)}{2}H(\delta)^{p(r)/2-1} \int_{ \mathbb{T}^2 } \nabla f_\delta(Z^\e)(\sigma_k(\Phi^\e)-\sigma_k(\Phi)) \ud x\ud W^k\\
	&\ \ + \frac12\sum_k\int^t_0 \frac{p(r)}{2}H(\delta)^{p(r)/2-1} \int_{ \mathbb{T}^2 } (\sigma_k(\Phi^\e)-\sigma_k(\Phi))\cdot D^2f_\delta(Z^\e)(\sigma_k(\Phi^\e)-\sigma_k(\Phi))\ud x\ud r\\
	&\ \ + \frac12\sum_k\int^t_0 \frac{p(r)(p(r)-2)}{4}H(\delta)^{p(r)/2-2} \left(\int_{ \mathbb{T}^2 } \nabla f_\delta(Z^\e)(\sigma_k(\Phi^\e)-\sigma_k(\Phi)) \ud x\right)^2 \ud r\\
	&\ \ + \int^t_0 \frac{p'(r)}{2}\log(H(\delta))H(\delta)^{p(r)/2} \ud r\\
	&=: A_u +A_{stoch} +A_{second-order-1} +A_{second-order-2} +A_{log} .
	\end{align*}
	We aim at controlling the (square of) the $L^2_\omega(C_t)$ norm in the equation above, but before doing this, we want to get rid of the log-Lipschitz dependency coming from $u^\e(\Phi^\e)-u(\Phi)$, which would otherwise cause problems: for this we will use the term $A_{log}$, which comes from $p'$. We start splitting $A_u$ as follows:
	\begin{align*}
	A_u&\le \int^t_0 \frac{p(r)}{2}H(\delta)^{p(r)/2-1} \int_{\ \mathbb{T}^2 } |u^\e(\Phi^\e)-u(\Phi)| \ud x \ud r\\	
	&\le \int^t_0 \frac{p(r)}{2}H(\delta)^{p(r)/2-1} \int_{\ \mathbb{T}^2 }\int_{\ \mathbb{T}^2 } |K^\e(\Phi^\e(x)-\Phi^\e(y))-K(\Phi(x)-\Phi(y))||\xi_0(x)| \ud x\ud y\ud r\\
	&\le \|\xi_0\|_{L^\infty_x}\int^t_0 \frac{p(r)}{2}H(\delta)^{p(r)/2-1} \int_{\ \mathbb{T}^2 }\int_{\ \mathbb{T}^2 } |K^\e(\Phi^\e(x)-\Phi^\e(y))-K(\Phi^\e(x)-\Phi^\e(y))| \ud x\ud y\ud r\\
	&\ \ +\|\xi_0\|_{L^\infty_x}\int^t_0 \frac{p(r)}{2}H(\delta)^{p(r)/2-1} \int_{\ \mathbb{T}^2 }\int_{\ \mathbb{T}^2 } |K(\Phi^\e(x)-\Phi^\e(y))-K(\Phi(x)-\Phi^\e(y))| \ud x\ud y\ud r\\
	&\ \ +\|\xi_0\|_{L^\infty_x}\int^t_0 \frac{p(r)}{2}H(\delta)^{p(r)/2-1} \int_{\ \mathbb{T}^2 }\int_{\ \mathbb{T}^2 } |K(\Phi(x)-\Phi^\e(y))-K(\Phi(x)-\Phi(y))| \ud x\ud y\ud r \\
	&=: A_{u1}+A_{u2}+A_{u3}.
	\end{align*}
	Leaving the term $A_{u1}$ for later, we estimate $A_{u2}$ and $A_{u3}$. For the term $A_{u2}$, we make the change of variable $y'=\Phi^\e(y)$ and we use the measure preserving property of $\Phi^\e$, the log-Lipschitz property associated with $K$ and Jensen inequality for the concave function $\gamma$ defined in \eqref{gamma_log}. We get:
	\begin{align*}
	A_{u2} = &\|\xi_0\|_{L^\infty_x}\int^t_0 \frac{p(r)}{2}H(\delta)^{p(r)/2-1} \int_{\ \mathbb{T}^2 }\int_{\ \mathbb{T}^2 } |K(\Phi^\e(x)-\Phi^\e(y))-K(\Phi(x)-\Phi^\e(y))| \ud x\ud y\ud r\\
	\le &C'\|\xi_0\|_{L^\infty_x}\int^t_0 \frac{p(r)}{2}H(\delta)^{p(r)/2-1} \int_{\ \mathbb{T}^2 }\gamma(|Z^\e(x)|) \ud x\ud r\\
	\le &C'\|\xi_0\|_{L^\infty_x}\int^t_0 \frac{p(r)}{2}H(\delta)^{p(r)/2-1} \gamma(H(\delta)) \ud r.
	\end{align*}
	For the term $A_{u3}$ we proceed similarly with the change of variable $x'=\Phi(x)$, getting
	\begin{align*}
	&A_{u3} \le C'\|\xi_0\|_{L^\infty_x}\int^t_0 \frac{p(r)}{2}H(\delta)^{p(r)/2-1} \gamma(H(\delta)) \ud r.
	\end{align*}
	Hence we can bound $A_{u2}+A_{u3}+A_{log}$ with
	\begin{align*}
	A_{u2}+A_{u3}+A_{log} \le \frac12 \int^t_0  2C'\|\xi_0\|_{L^\infty_x} p(r)H(\delta)^{p(r)/2-1} \gamma(H(\delta)) + p'(r)\log(H(\delta))H(\delta)^{p(r)/2}] \ud r.
	\end{align*}
	Now we choose $\lambda= 2C'\|\xi_0\|_{L^\infty_x}$, which gives (recall the definition of $\gamma$ in \eqref{gamma_log})
	\begin{align*}
	p'(r)\log(H(\delta))H(\delta) +2C'\|\xi_0\|_{L^\infty_x}p(r)\gamma(H(\delta))= p(r)(\log(H(\delta))H(\delta) +\gamma(H(\delta))).
	\end{align*}
	We use the following inequality for $\gamma$, valid for $r$ in a bounded interval $[0,R]$:
	\begin{align*}
	\gamma(r)+r\log r \le Cr,
	\end{align*}
	for some $C$ depending only on $R$. Hence we get
	\begin{align*}
	A_{u2}+A_{u3}+A_{log} \le C\|\xi_0\|_{L^\infty_x} \int^t_0 p(r)H(\delta)^{p(r)/2} \ud r,
	\end{align*}
	and so
	\begin{align*}
	&H(\delta)_t^{p(t)/2} -H(\delta)_0^{p/2}
	\le A_{u1} +A_{stoch} +A_{second-order-1}+A_{second-order-2} +C\|\xi_0\|_{L^\infty_x} \int^t_0 p(r)H(\delta)^{p(r)/2} \ud r.
	\end{align*}
	At this point we control the square of the $L^2_\omega(C_t)$ norm of each addend of the right hand side. For the term $A_{u1}$, we make the change of variable $x'=\Phi^\e(x)$, $y'=\Phi^\e(y)$ and use the measure preserving property of $\Phi^\e$, getting, via Young inequality (with exponents $p(r)/2$ and its conjugate),
	\begin{align*}
	\E \sup_{ s \in [0,t]}A_{u1}^2 \le & \E\|\xi_0\|_{L^\infty_x}^2\int^t_0 \frac{p(r)^2}{4}H(\delta)^{p(r)-2} \left(\int_{\ \mathbb{T}^2 }\int_{\ \mathbb{T}^2 } |K^\e(\Phi^\e(x)-\Phi^\e(y))-K(\Phi^\e(x)-\Phi^\e(y))| \ud x\ud y\right)^2\ud r\\
	\le & C\E \|\xi_0\|_{L^\infty_x}^2 \int^t_0 p(r)^2(\|K^\e-K\|_{L^1_x}^{p(r)} + H(\delta)^{p(r)}) \ud r\\
	\le & C\|\xi_0\|_{L^\infty_x}(e^{C\|\xi_0\|_{L^\infty_x}}-1)\|K^\e-K\|_{L^1_x}^p +C\|\xi_0\|_{L^\infty_x}^2e^{C\|\xi_0\|_{L^\infty_x}}\int^t_0\E\sup_{ s \in[0,r]}H(\delta)_{ s }^{p( s )} \ud r.
	\end{align*}
	For the term $A_{stoch}$, using Burkholder-Davis-Gundy inequality and the Lipschitz property of $\sigma_k$ we get
	\begin{align*}
	\E \sup_{ s \in [0,t]}A_{stoch}^2 \le & C\E\int^t_0 p(r)^2H(\delta)^{p(r)-2} \int_{ \mathbb{T}^2 } \sum_k|\sigma_k(\Phi^\e(x))-\sigma_k(\Phi(x))|^2 \ud x\ud r\\
	\le & C\sum_k\|\sigma_k\|_{W^{1,\infty}_x}^2 \int^t_0 p(r)^2 \E H(\delta)^{p(r)}_r \ud r\\
	\le & Ce^{C\|\xi_0\|_{L^\infty_x}}\sum_k\|\sigma_k\|_{W^{1,\infty}_x}^2 \int^t_0 \E\sup_{ s \in[0,r]}H(\delta)^{p( s )}_{ s } \ud r.
	\end{align*}
	For the term $A_{second-order-1}$, using again the Lipschitz property of $\sigma_k$ we get
	\begin{align*}
	\E \sup_{ s \in [0,t]}A_{second-order-1}^2 \le & C\E\int^t_0 p(r)^2 H(\delta)^{p(r)-2} \left(\int_{\ \mathbb{T}^2 } \sum_k|\sigma_k(\Phi^\e(x))-\sigma_k(\Phi(x))|^2 |D^2f_\delta(Z^\e)|\ud x\right)^2\ud r\\
	\le & C(\sum_k\|\sigma_k\|_{W^{1,\infty}_x}^2)^2 \E\int^t_0 p(r)^2 H(\delta)^{p(r)-2} \left(\int_{\ \mathbb{T}^2 } |Z^\e| |Z^\e|^{-1}|Z^\e| \ud x\right)^2\ud r\\
	\le & Ce^{C\|\xi_0\|_{L^\infty_x}} (\sum_k\|\sigma_k\|_{W^{1,\infty}_x}^2)^2 \E\int_{0}^{t}\sup_{ s \in[0,r]}H(\delta)^{p( s )}_{ s } \ud r.
	\end{align*}
	Similarly for the term $A_{second-order-2}$ we get
	\begin{align*}
	\E \sup_{ s \in [0,t]}A_{second-order-2}^2 \le & C\E\int^t_0 p(r)^2(p(r)-2)^2H(\delta)^{p(r)-4} \left(\sum_k\left(\int_{ \mathbb{T}^2 } |\sigma_k(\Phi^\e(x))-\sigma_k(\Phi(x))| \ud x\right)^2\right)^2 \ud r\\
	\le & C(\sum_k\|\sigma_k\|_{W^{1,\infty}_x}^2)^2\E\int^t_0 p(r)^2(p(r)-2)^2H(\delta)^{p(r)-4} \left(\int_{ \mathbb{T}^2 } |Z^\e| \ud x\right)^4 \ud r\\
	\le & Ce^{C\|\xi_0\|_{L^\infty_x}}(\sum_k\|\sigma_k\|_{W^{1,\infty}_x}^2)^2\E \int_{0}^{t}\sup_{ s \in[0,r]}H(\delta)^{p( s )}_{ s } \ud r.
	\end{align*}
	Putting all together, we obtain, for some $C=C(\sum_k\|\sigma_k\|_{W^{1,\infty}_x}^2)$ (possibly depending also on $p$ and $T$),
	\begin{align*}
	\E[\sup_{ s \in [0,t]}H(\delta)_{ s }^{p( s )}] \le & \E[H(\delta)_0^p] 
	+ C\|\xi_0\|_{L^\infty_x}(e^{C\|\xi_0\|_{L^\infty_x}}-1)\|K^\e-K\|_{L^1_x}^p \\
	& + C(1+\|\xi_0\|_{L^\infty_x}^2)e^{C\|\xi_0\|_{L^\infty_x}}\int^t_0\E\sup_{ s \in[0,r]}H(\delta)_{ s }^{p( s )} \ud r.
	\end{align*}
	Applying first Young inequality and then letting $\delta\rightarrow0$ ($f_\delta$ converges to $f$ uniformly), we obtain
	\begin{align*}
	&\E[\sup_{ s \in [0,t]}H_{ s }^{p( s )}]\le C\|\xi_0\|_{L^\infty_x}e^{C\|\xi_0\|_{L^\infty_x}+C\|\xi_0\|_{L^\infty_x}^2+\exp[C\|\xi_0\|_{L^\infty_x}]}\|K^\e-K\|_{L^1_x}^p.
	\end{align*}
	Since $H$ is uniformly bounded (as $\Phi$ and $\Phi^\e$ take values on $\ \mathbb{T}^2 $) and $p(t)$ is increasing in $t$, then $H^{p( s )}\ge c H^{p(t)}$ for any $s\le t$, for some $c=c(\|\xi_0\|_{L^\infty})>0$ (depending also on $T$ and $p$). Therefore we conclude
	\begin{align*}
	\E[\sup_{ s \in [0,t]}H_{ s }^{p}] \le & \E[\sup_{ s \in [0,t]}H_{ s }^{p(t)}]^{p/p(t)}
	\le C\E[\sup_{ s \in [0,t]}H_{ s }^{p( s )}]^{p/p(t)} 
	\le  C\|K^\e-K\|_{L^1_x}^{p^2/p(t)} = C\|K^\e-K\|_{L^1_x}^{pe^{-\lambda t}}.
	\end{align*}
\end{proof}

\section{Vortex approximation for Euler equations}\label{Convergence_Section}

\begin{theorem}\label{thm: convergence_euler}
	Let $M \in \R_{+}$ and $\xi_0 \in L^\infty(\ \mathbb{T}^2 )$ such that $\lVert \xi_0 \rVert_{L^\infty} \leq M$. Let $(x^{i,N})_{1\leq i \leq N} \subset \T^2$ and $(\xi^{i,N})_{1\leq i \leq N} \subset \mathbb{R}$ such that  $S_0^N := \frac{1}{N} \sum_{i=1}^{N} \xi^{i,N} \delta_{x^{i,N}}$ is in $\mathcal{M}_M$. Assume that
	\begin{equation}\label{conv_initial}
	W_1(S^N_0, \xi_0) =: \zeta_N \to 0, \quad \mbox{as } N \to \infty .
	\end{equation}
	Call $S_t^{N,\e}$ the empirical measure associated to the system \eqref{eq: regular particles} with initial condition $(x^{i,N})_{1\leq i \leq N}$ and $\xi_t$ the solution to the Euler equation \eqref{eq:vort_intro} starting form $\xi_0$. Then, taking the approximation $\e(N) = o\left((-\log(\zeta_N))^{-(4 ( 2 + \delta ))^{-1}} \right)$, the following convergence holds true, on every time interval $[0, T]$, 
	\begin{equation*}
	d_1(S^{N,\e(N)}, \xi) \to 0, \quad \mbox{as } N\to\infty.
	\end{equation*}
\end{theorem}

\begin{proof}
	Let $\xi^{\e}$ be a solution to equation \eqref{regular euler}. We split
	\begin{equation}\label{zero_last_thm}
	d_1(S^{N,\e}, \xi) \leq d_1(S^{N,\e}, \xi^\e) + d_1(\xi^\e, \xi).
	\end{equation}
	We will obtain the estimate of the two terms on the right-hand side as a consequence of Corollary \ref{cor: convergence regular case} and Theorem \ref{convergence_approximation} respectively.
	
	For the second term in the right-hand side of \eqref{zero_last_thm}, using the definition of $d_1$, we have
	\begin{align*}
	d_1(\xi^\e, \xi) 
	\leq & \E\left[\sup_{t\in[0,T]}\int\left\vert \Phi_t^\e(x) - \Phi_t(x) \right\vert \vert \xi_0(x)\vert dx\right].
	\end{align*}	
	Recall that the initial condition $\xi_0$ is deterministic and in $L^\infty(\ \mathbb{T}^2 )$. Hence from Theorem \ref{convergence_approximation} we obtain
	\begin{equation}\label{second_last_thm}
	d_1(\xi^\e, \xi) \leq C(M) \lVert K^\e - K \rVert_{L^1},
	\end{equation}	
	which goes to $0$ by Lemma \ref{lem: properties mollified kernel}.
	For the first term in the right-hand side of \eqref{zero_last_thm}, we have the following estimate from Corollary \ref{cor: convergence regular case},
	\begin{equation*}\label{first_last_thm}
		d_1(S^{N,\e}, \xi^\e) 
		\leq d_4(S^{N,\e}, \xi^\e) 
		\leq 2e^{cT} \Gamma W_1(S_0^N, \xi_0).
	\end{equation*}
	From the definition of $c$ and $\Gamma$ and Lemma \ref{lem: properties mollified kernel: explosion} we have that for any $\delta > 0$,
		\begin{align*}
		\Gamma  e^{cT}
		\sim
		C \| D^2 K^\e  \|_{C^0} 
		e^{ C \| DK^\e  \|_{C^0} ^{p}} e^{ C \| DK^\e  \|_{C^0} ^{p^2/(2p-4)}}
		\sim C \e ^ { - ( 3 + \delta ) } e^{ C\e ^ {- 4 ( 2 + \delta ) }}
		\sim e^{ C\e ^ {- 4 ( 2 + \delta ) }},
		\qquad
		\mbox{as } \e  \to 0,
		\end{align*}		
	with $C = C(T, \alpha, \sigma)$ (note that $p=4$ minimizes $p^2/(2p-4)$). We conclude by taking $\e(N) = o\left((-\log(\zeta_N))^{-(4 ( 2 + \delta ))^{-1}} \right)$, so that
	\begin{equation*}
	\e(N) \to 0
	\quad
	\mbox{and}
	\quad
	e^{ \e(N) ^ {- 7 ( 2 + \delta ) } } W_1(S_0^N, \xi_0) \to 0,
	\qquad
	\mbox{as } N \to \infty.
	\end{equation*}
%
	
\end{proof}

We show now that starting from a bounded initial vorticity it is always possible to construct a sequence of empirical measures that satisfies \eqref{conv_initial}.
\begin{lemma}
	Let $\xi_0 \in L^{\infty}(\T^2)$ with $\| \xi_0 \|_{L^{\infty}} \leq M < \infty$. 
	There exist families $(\xi^{i})_{i \in \N} \subset [-M, M]$ and $(x^i)_{i\in \N} \subset \T^2$ such that
	\begin{equation*}
	W_1 \left( \frac{1}{N} \sum_{i=1}^{N} \xi^{i} \delta_{x^i}, \xi_0 \right) \to 0,
	\qquad
	N\to \infty.
	\end{equation*}
\end{lemma}
 
	\begin{proof}
		There exist two non negative functions $\xi_0^{+}, \xi_0^{-} \in L^{\infty}(\T^2)$ such that $\xi_0 = \xi_0^{+} - \xi_0^{-}$, Lebesgue-almost surely. We define
		\begin{equation*}
		\mu_0(dm, dx) := \frac{1}{M} \delta_{M}(dm) \xi_0^{+}(x)dx + \frac{1}{M} \delta_{-M}(dm) \xi_0^{-}(x)dx
		\in 
		\mathcal{P}([-M, M] \times \T^2).
		\end{equation*}
		On an abstract probability space $(\Omega, \mathcal{F}, \mathbb{P})$, we consider an independent and identically distributed sequence of random variables $(M^i, X^i)_{i\in \N}$ with law $\mu_0$.
		
		Let $\varphi \in BL_1(\T^2)$, then $(m, x) \mapsto \frac{m\varphi(x)}{M+1} \in BL_1([-M, M] \times \T^2)$. Hence, we have $\mathbb{P}$-a.s.,
		\begin{equation*}
		W^1 \left( \frac{1}{N} \sum_{i=1}^{N} M^i \delta_{X^i}, \xi_0 \right) 
		=
		\sup_{\varphi \in BL_1}
		\int_{[-M, M] \times \T^2} 
		m\varphi(x) 
		\left(
		\frac{1}{N} \sum_{i=1}^{N} 
		\delta_{( M^i , X^i)} - \mu_0 
		\right) (dm, dx).
		\end{equation*}
		By the law of large numbers, the right-hand side goes to zero almost surely. Thus, for any $\omega$ in a set of full measure we have that the lemma is satisfied for the families $(\xi^i, x^i)_{i\in \N} := (M^i(\omega), X^i(\omega))_{i \in \N}$.
	\end{proof}

\bibliographystyle{abbrv}
\bibliography{bibliography}

\begin{thebibliography}{10}

\bibitem{BarBiaFla2013}
D.~Barbato, L.~A. Bianchi, F.~Flandoli, and F.~Morandin.
\newblock A dyadic model on a tree.
\newblock {\em J. Math. Phys.}, 54(2):021507, 20, 2013.

\bibitem{beck2014stochastic}
L.~Beck, F.~Flandoli, M.~Gubinelli, and M.~Maurelli.
\newblock Stochastic {ODE}s and stochastic linear {PDE}s with critical drift:
  regularity, duality and uniqueness.
\newblock {\em Electron. J. Probab.}, 24:Paper No. 136, 72, 2019.

\bibitem{BesCogFla2017}
H.~Bessaih, M.~Coghi, and F.~Flandoli.
\newblock Mean field limit of interacting filaments and vector valued
  non-linear {PDE}s.
\newblock {\em J. Stat. Phys.}, 166(5):1276--1309, 2017.

\bibitem{BesCogFla2019}
H.~Bessaih, M.~Coghi, and F.~Flandoli.
\newblock Mean field limit of interacting filaments for 3{D} {E}uler equations.
\newblock {\em J. Stat. Phys.}, 174(3):562--578, 2019.

\bibitem{BesGarSch2016}
H.~Bessaih, M.~J. Garrido-Atienza, and B.~Schmalfuss.
\newblock Stochastic shell models driven by a multiplicative fractional
  {B}rownian-motion.
\newblock {\em Phys. D}, 320:38--56, 2016.

\bibitem{Bia2013}
L.~A. Bianchi.
\newblock Uniqueness for an inviscid stochastic dyadic model on a tree.
\newblock {\em Electron. Commun. Probab.}, 18:no. 8, 12, 2013.

\bibitem{BiaFla2020}
L.~A. Bianchi and F.~Flandoli.
\newblock Stochastic navier-stokes equations and related models, 2020.

\bibitem{BiaMor2017}
L.~A. Bianchi and F.~Morandin.
\newblock Structure function and fractal dissipation for an intermittent
  inviscid dyadic model.
\newblock {\em Comm. Math. Phys.}, 356(1):231--260, 2017.

\bibitem{Brzezniak_Flandoli_Maurelli}
Z.~Brze\'{z}niak, F.~Flandoli, and M.~Maurelli.
\newblock Existence and uniqueness for stochastic 2{D} {E}uler flows with
  bounded vorticity.
\newblock {\em Arch. Ration. Mech. Anal.}, 221(1):107--142, 2016.

\bibitem{chorin1982}
A.~J. Chorin.
\newblock The evolution of a turbulent vortex.
\newblock {\em Comm. Math. Phys.}, 83(4):517--535, 1982.

\bibitem{coghi_flandoli}
M.~Coghi and F.~Flandoli.
\newblock Propagation of chaos for interacting particles subject to
  environmental noise.
\newblock {\em Ann. Appl. Probab.}, 26(3):1407--1442, 2016.

\bibitem{CriFlaHol2019}
D.~Crisan, F.~Flandoli, and D.~D. Holm.
\newblock Solution {P}roperties of a 3{D} {S}tochastic {E}uler {F}luid
  {E}quation.
\newblock {\em J. Nonlinear Sci.}, 29(3):813--870, 2019.

\bibitem{DelFlaVin2014}
F.~Delarue, F.~Flandoli, and D.~Vincenzi.
\newblock Noise prevents collapse of {V}lasov-{P}oisson point charges.
\newblock {\em Comm. Pure Appl. Math.}, 67(10):1700--1736, 2014.

\bibitem{Do}
R.~L. Dobru\v{s}in.
\newblock Vlasov equations.
\newblock {\em Funktsional. Anal. i Prilozhen.}, 13(2):48--58, 96, 1979.

\bibitem{Fla2018}
F.~Flandoli.
\newblock Weak vorticity formulation of 2{D} {E}uler equations with white noise
  initial condition.
\newblock {\em Comm. Partial Differential Equations}, 43(7):1102--1149, 2018.

\bibitem{FlaGubPri2010}
F.~Flandoli, M.~Gubinelli, and E.~Priola.
\newblock Well-posedness of the transport equation by stochastic perturbation.
\newblock {\em Invent. Math.}, 180(1):1--53, 2010.

\bibitem{FlaGubPri2011}
F.~Flandoli, M.~Gubinelli, and E.~Priola.
\newblock Full well-posedness of point vortex dynamics corresponding to
  stochastic 2{D} {E}uler equations.
\newblock {\em Stochastic Process. Appl.}, 121(7):1445--1463, 2011.

\bibitem{FlaLuo2019}
F.~Flandoli and D.~Luo.
\newblock High mode transport noise improves vorticity blow-up control in 3d
  navier-stokes equations, 2019.

\bibitem{FlandoliSaal}
F.~Flandoli and M.~Saal.
\newblock m{SQG} equations in distributional spaces and point vortex
  approximation.
\newblock {\em J. Evol. Equ.}, 19(4):1071--1090, 2019.

\bibitem{FouHauMis2014}
N.~Fournier, M.~Hauray, and S.~Mischler.
\newblock Propagation of chaos for the 2{D} viscous vortex model.
\newblock {\em J. Eur. Math. Soc. (JEMS)}, 16(7):1423--1466, 2014.

\bibitem{geldhauser2018limit}
C.~Geldhauser and M.~Romito.
\newblock Limit theorems and fluctuations for point vortices of generalized
  euler equations.
\newblock {\em arXiv preprint arXiv:1810.12706}, 2018.

\bibitem{GessMaurelli}
B.~Gess and M.~Maurelli.
\newblock Well-posedness by noise for scalar conservation laws.
\newblock {\em Comm. Partial Differential Equations}, 43(12):1702--1736, 2018.

\bibitem{GooHouLow1990}
J.~Goodman, T.~Y. Hou, and J.~Lowengrub.
\newblock Convergence of the point vortex method for the {$2$}-{D} {E}uler
  equations.
\newblock {\em Comm. Pure Appl. Math.}, 43(3):415--430, 1990.

\bibitem{grotto2019central}
F.~Grotto and M.~Romito.
\newblock A central limit theorem for gibbsian invariant measures of 2d euler
  equation.
\newblock {\em arXiv preprint arXiv:1904.01871}, 2019.

\bibitem{HofLeaNil2019+}
M.~Hofmanova, J.-M. Leahy, and T.~Nilssen.
\newblock On a rough perturbation of the navier-stokes system and its vorticity
  formulation.
\newblock {\em arXiv:1902.09348}, 2019.

\bibitem{JabWan2018}
P.-E. Jabin and Z.~Wang.
\newblock Quantitative estimates of propagation of chaos for stochastic systems
  with {$W^{-1,\infty}$} kernels.
\newblock {\em Invent. Math.}, 214(1):523--591, 2018.

\bibitem{Ku90}
H.~Kunita.
\newblock {\em Stochastic flows and stochastic differential equations},
  volume~24 of {\em Cambridge Studies in Advanced Mathematics}.
\newblock Cambridge University Press, Cambridge, 1997.
\newblock Reprint of the 1990 original.

\bibitem{MarPul1982}
C.~Marchioro and M.~Pulvirenti.
\newblock Hydrodynamics in two dimensions and vortex theory.
\newblock {\em Comm. Math. Phys.}, 84(4):483--503, 1982.

\bibitem{Marchioro_Pulvirenti}
C.~Marchioro and M.~Pulvirenti.
\newblock {\em Vortex methods in two-dimensional fluid dynamics}, volume 203 of
  {\em Lecture Notes in Physics}.
\newblock Springer-Verlag, Berlin, 1984.

\bibitem{MarPul1994}
C.~Marchioro and M.~Pulvirenti.
\newblock {\em Mathematical theory of incompressible nonviscous fluids},
  volume~96 of {\em Applied Mathematical Sciences}.
\newblock Springer-Verlag, New York, 1994.

\bibitem{Sch1996}
S.~Schochet.
\newblock The point-vortex method for periodic weak solutions of the 2-{D}
  {E}uler equations.
\newblock {\em Comm. Pure Appl. Math.}, 49(9):911--965, 1996.

\bibitem{Szn1984}
A.-S. Sznitman.
\newblock Nonlinear reflecting diffusion process, and the propagation of chaos
  and fluctuations associated.
\newblock {\em J. Funct. Anal.}, 56(3):311--336, 1984.

\bibitem{Yud1963}
V.~I. Yudovi\v{c}.
\newblock Non-stationary flows of an ideal incompressible fluid.
\newblock {\em \v{Z}. Vy\v{c}isl. Mat. i Mat. Fiz.}, 3:1032--1066, 1963.

\end{thebibliography}

\end{document}